\documentclass[reqno,oneside,12pt]{amsart}

\usepackage[T1]{fontenc}
\usepackage{times,mathptm}
\usepackage{amssymb,epsfig,verbatim,xypic}
\usepackage[utf8]{inputenc}
\usepackage{amsmath}
\usepackage{enumitem}
\usepackage[normalem]{ulem}
\usepackage{hyperref}
\def\vs{\vspace{0.1cm}}

\theoremstyle{plain}

\newtheorem{thm}{Theorem}[section]

\newtheorem{cor}[thm]{Corollary}
\newtheorem{pro}[thm]{Proposition}
\newtheorem{lem}[thm]{Lemma}
\newtheorem{proposition-principale}[thm]{Proposition principale}
\newtheorem{thm-principal}{Th\'eor\`eme principal}[section]

\theoremstyle{definition}
\newtheorem{defi}[thm]{Definition}

\newtheorem{eg}[thm]{Example}
\newtheorem{rem}[thm]{Remark}
\usepackage{xcolor}

\usepackage{cases}

\newenvironment{thm-A}
{{\vs \noindent \bf Theorem A.$\,$}\it}{\vs}

\newenvironment{thm-B}
{{\vs \noindent \bf Theorem B.$\,$}\it}{\vs}

\newenvironment{thm-C}
{{\vs \noindent \bf Theorem C.$\,$}\it}{\vs}

\newenvironment{thm-BB}
{{\vs \noindent \bf Theorem B'.$\,$}\it}{\vs}

\newenvironment{thm-AA}
{{\vs \noindent \bf Theorem A'.$\,$}\it}{\vs}

\def\C{\mathbf{C}}
\def\R{\mathbf{R}}
\def\Q{\mathbf{Q}}
\def\Z{\mathbf{Z}}

\def\bfk{{\mathbf{k}}}

 \def\jj{{\mathsf{j}}}

\def\bfe{{\mathbf{e}}}

\def\bfE{{\mathbf{E}}}
\def\Frob{{\mathrm{Fr}}}
\def\calg{\mathcal{G}}

\newcommand{\Hyp}{\mathbb{H}}
\newcommand{\ZZ}{{\mathcal{Z}}}
\newcommand{\NS}{{\mathrm{NS}}}

\newcommand{\Elp}{{\mathsf{pEl}}}

\newcommand{\El}{{\mathsf{El}}}

\newcommand{\eps}{\epsilon}

\newcommand{\id}{{\rm Id}}
\def\Cr{\mathsf{Cr}}
\def\Trans{\mathsf{T}}
\def\P{\mathbb{P}}
\def\bbA{\mathbb{A}}

\def\bbG{{\mathbb{G}}}
\def\Aut{\mathsf{Aut}}
\def\pAut{\mathsf{pAut}}

\def\fieldchar{{\mathrm{char}}}
\def\calT{\mathcal{T}}

\DeclareMathOperator{\arcosh}{arcosh}
\DeclareMathOperator{\arcsinh}{arcsinh}
\DeclareMathOperator{\Cone}{Cone}

 \def\Bir{{\mathsf{Bir}}}

 \def\PGL{{\sf{PGL}}}

\def\O{{\sf{O}}}

\def\GL{{\sf{GL}}}
\def\SL{{\sf{SL}}}
\def\Aff{{\sf{Aff}}}

\def\Ind{{\text{Ind}}}

\def\dist{{\sf{dist}}}

\def\fieldchar{{\rm{char}\,}}

\newcommand{\Stab}{\mathrm{Stab}}

\newcommand{\an}[1]{{\textcolor{red}{#1}}}

\addtocounter{section}{0}             
\numberwithin{equation}{section}       

\begin{document}

\setlength{\baselineskip}{0.53cm}        
%
%
\title[Generating normal subgroups of the Cremona group]{Elements generating a proper normal subgroup of the Cremona group}
\date{January 6, 2020}

\author{Serge Cantat, Vincent Guirardel, and Anne Lonjou}
\address{Serge Cantat, Vincent Guirardel, \\
Univ Rennes, CNRS, IRMAR - UMR 6625, F-35000 Rennes, France}
\email{serge.cantat@univ-rennes1.fr} \email{vincent.guirardel@univ-rennes1.fr}
 \address{Anne Lonjou, University of Basel}
 \email{anne.lonjou@unibas.ch}

\thanks{The third author acknowledges support from the french Academy of Sciences (Fondation del Duca), and the third author
from the Swiss National Science Foundation Grant ``Birational transformations of threefolds'' $200020\mathunderscore178807$. }

%
%

\maketitle
 

%
%
%
%

\begin{abstract}
Consider an algebraically closed field $\bfk$, and let $\Cr_2(\bfk)$ be the Cremona group of all birational transformations of the projective plane over~$\bfk$.
We characterize infinite order elements $g\in \Cr_2(\bfk)$ having a power $g^n$, $n\neq 0$, generating a proper normal subgroup of $\Cr_2(\bfk)$.
\end{abstract}



\section{Introduction}

Let $\bfk$ be a field. Let $\Cr_2(\bfk)$ be the group of birational 
transformations of the projective plane $\P^2_\bfk$; we shall call it the {\bf{Cremona group}}. 
If $f$ is an element of $\Cr_2(\bfk)$, one can write $f$ in homogeneous coordinates $[x:y:z]$ as
\begin{equation}
f[x:y:z]=[P:Q:R]
\end{equation}
where $P$, $Q$, and $R$ are three homogeneous polynomials of the same degree $d$ in $(x,y,z)$
with no common factor of positive degree: the integer $d$
is called the degree of $f$, and is denoted $\deg(f)$. 

The Cremona group acts by isometries on a hyperbolic space of infinite dimension $\Hyp_\infty$ (see Subsection \ref{par:ZZ}). 
There are three types of isometries of such a space, and therefore also three types of elements $f\in \Cr_2(\bfk)$ (see Section~\ref{par:IsomType}): 
\begin{enumerate}
\item {\bf{elliptic}} isometries correspond to birational transformations for which the sequence $(\deg(f^n))$ is bounded;
\item {\bf{parabolic}} isometries correspond to a polynomial growth of $(\deg(f^n))$: either $\deg(f^n)$ grows linearly and  $f$ is
called a {\bf{Jonqui\`eres twist}}, or $\deg(f^n)$ grows quadratically and  $f$ is called a {\bf{Halphen twist}};
\item {\bf{loxodromic}} isometries correspond to birational transformations for which $\deg(f^n)$ grows exponentially fast as $n$
goes to $+\infty$. 
\end{enumerate}
The name for Jonqui\`eres and Halphen twists come from the following properties. If $\bfk$ is algebraically closed, such a twist $f\colon \P^2_\bfk\dasharrow\P^2_\bfk$ 
preserves a unique pencil of curves, of genus $0$ or $1$ respectively. Moreover, every pencil of rational curves is equivalent, via a birational 
change of coordinates, to the pencil of lines
through the point $[0:0:1]$, and the group of birational transformations preserving this pencil is known as the Jonqui\`eres group. 
And every pencil of curves of genus $1$ is equivalent to a Halphen pencil (see Section~\ref{par:Halphen}). 

A {\bf{monomial}} transformation is a birational transformation that, in affine coordinates $(x,y)=[x:y:1]$, can be written
$f(x,y)=(x^ay^b, x^cy^d)$ for some integers $a$, $b$, $c$, and $d$. The group of all monomial transformations is 
isomorphic to $\GL_2(\Z)$. 

In a group $G$, we say that an element $h$ {\bf {generates a proper normal subgroup}} if the smallest normal subgroup  
containing $h$ is not equal to $G$.

 \medskip

\begin{thm-A}
Let $\bfk$ be an algebraically closed field of characteristic $0$.
Let $f$ be an element of $\Cr_2(\bfk)$ of infinite order. The following properties are equivalent. 
\begin{itemize}
\item[(a)] For some  $n\neq 0$, $f^n$ generates a proper normal subgroup of $\Cr_2(\bfk)$. 
\item[(b)] The birational transformation $f$ is  a Halphen twist or a loxodromic element that is not conjugate to a monomial transformation.
\end{itemize}
\end{thm-A}

A version of this theorem in positive characteristic is given in Section~\ref{par:charpos}, Theorem~A'.
The proof of this theorem is a combination of several recent results together with one new input. Indeed, the heart of
our article is Theorem~B:

\medskip

\begin{thm-B}\label{thm_Halphen}  
There exists a positive integer $n_0$ with the following property. Let $\bfk$ be a field. If $g\in \Cr_2(\bfk)$ is a Halphen twist, then
 $g^{n_0}$ generates a proper normal subgroup of $\Cr_2(\bfk)$. 
Moreover, this normal subgroup is a free product of free abelian groups of rank $\leq 8$.
\end{thm-B}

\medskip

The proof of Theorem~B combines algebraic geometry with 
geometric group theory, as in \cite{DGO} and \cite{Cantat-Lamy}. 
In Theorem~B, $n_0$ does not depend on the Halphen twist $g$, but in Theorem~A, the integer $n$ must depend on $f$: see Theorem~\ref{thm:non-uniform}. We refer to  Theorem~B', stated in Section~\ref{par:conclusion}, for a slightly stronger result.

\medskip

\subsection*{Acknowledgement} Thanks to Julie D\'eserti for interesting discussions regarding Halphen pencils (parts of \S~\ref{par:Halphen}
are based on a joint work of the first author with Julie) and St\'ephane Lamy for numerous discussions on normal subgroups of $\Cr_2(\bfk)$.
We are grateful to the referees for numerous interesting remarks.

\section{Rational surfaces and hyperbolic spaces}

 \subsection{The lattice $\Z^{1,n}$} \label{par:Z1n}
 
Let $n$ be a (maybe infinite) cardinal number; in what follows $n$ will be either a positive integer or the cardinality of the 
field $\bfk$. 

Denote by $\R^n$ a real Hilbert space of dimension $n$, with a fixed
orthonormal Hilbert basis $(\bfe_j)_{j\in J}$, where $J$ has cardinal $n$ (which may be uncountable).
Then, consider a one dimensional space $\R\bfe_0$, and define $\R^{1,n}$ to be the direct sum $\R\bfe_0\oplus \R^n$, together 
with the (indefinite) scalar product defined by
\begin{equation}
u\cdot v = a_0b_0-\sum_{j\in J} a_jb_j
\end{equation}
for every pair of vectors $u=a_0\bfe_0+\sum_j a_j\bfe_j$ and $v=b_0\bfe_0+\sum_j b_j\bfe_j$ with coefficients satisfying $\sum_j a_j^2<+\infty$ and $\sum_j b_j^2<+\infty$. We shall also set $u^2=u\cdot u$.

The subset $\Hyp_n\subset \R^{1,n}$ is defined as the set of vectors $u$ in $\R^{1,n}$ with $u^2=1$ and 
$u \cdot  \bfe_0 >0$. With the distance $\dist$  defined by 
the formula
\begin{equation}
\cosh (\dist(u,v) )= u\cdot v,
\end{equation}
$\Hyp_n$ is isometric to the classical hyperbolic space of dimension $n$.
Its  boundary $\partial \Hyp_n$ can be identified with the isotropic rays $\R\xi$ (where $\xi$ is any non-zero vector
in $\R^{1,n}$ with $\xi^2  =0$). 

The lattice $\Z^{1,n}$ is, by definition, the subset of vectors $u= \sum_i a_i\bfe_i\in \R^{1,n}$ with integer coefficients $(a_i)$.
The family  $(\bfe_i)_{i\in I}$ is a basis of $\Z^{1,n}$ 
and the quadratic form associated to the scalar product satisfies
\begin{equation}\label{eq:intersection-basis}
  \bfe_0 \cdot \bfe_0  = 1, \;    \bfe_j \cdot \bfe_j=-1 \; {\text{for}} \; 1\leq j \leq n, \; {\text{and}} \;   \bfe_i \cdot \bfe_j =0 \;{\text{if}} \; i\neq j.
\end{equation}
When $n$ is finite, this is the standard odd, unimodular quadratic form of signature $(1,n)$. 
To simplify the exposition, we shall write abusively $\R^{1,\infty}$ instead of $\R^{1,n}$ for any infinite $n$.

 \subsection{Blow-ups of the plane} 
 
Now, let $X$ be a rational surface that is obtained from $\P^2_\bfk$ by $n$ successive blow-ups; thus, 
$X$ comes with a birational morphism $\pi\colon X\to \P^2$.  
Denote by $\NS(X)$ the  N\'eron-Severi group of $X$, by $\bfe_0$ the class of a line and by $\bfe_i$, $1\leq i\leq n$, 
the classes of the exceptional divisors (more precisely, $\bfe_0$ is the pull-back of the class of a line by $\pi$, 
and each $\bfe_i$ is the pull-back in $\NS(X)$ of the class of an exceptional divisor). The $\bfe_i$ form 
a basis of $\NS(X)$ and the intersection 
products between the $\bfe_i$ are exactly as in Equation~\eqref{eq:intersection-basis}. Thus, $\NS(X)$ is isometric 
to the lattice $\Z^{1,n}$, by a unique isometry identifying the two bases $(\bfe_i)$. Viewed in $\NS(X;\R)=\NS(X)\otimes \R$, 
the hyperbolic space $\Hyp_n$ will be denoted 
by $\Hyp_X$.

 \subsection{Infinitely many blow-ups}\label{par:ZZ}
 
 Blowing up all possible points of $\P^2_\bfk$, including infinitely near ones, 
 one gets a nested family of rational surfaces $\pi\colon X\to \P^2_\bfk$.
The inductive limit of their N\'eron-Severi groups
$\NS(X)$ is a well defined, infinite dimensional $\Z$-module $\ZZ$;  by definition, $\ZZ$ is the {\bf{Picard-Manin space}}. 
It comes with an intersection form and  a natural basis $(\bfe_i)$, where $\bfe_0$ is the class of a line in the plane, and 
 each $\bfe_j$, $j\neq 0$, is the class of the exceptional divisor of a point (in some rational surface $X\to \P^2_\bfk$). Their
 relative intersections satisfy the Equation~\eqref{eq:intersection-basis}.
 We refer to \cite{Cantat:Survey, Dolgachev:book, Manin:book} for a detailed account on this construction
(and the definition of the bubble space indexing the elements $\bfe_j$ of that basis). 
With the notation of Section~\ref{par:Z1n}, $\ZZ$ is isometric to a lattice $\Z^{1,\infty}\subset \R^{1,\infty}$;
here, $\R^{1,\infty}$ is the direct sum of $\R \bfe_0$ and of the $\ell^2$-completion of $\oplus_{i\neq 0}\R\bfe_i$. 

Now assume $\bfk$ to be algebraically closed.
 Since all points have been  blown-up to construct $\ZZ$, 
all indeterminacy points of birational transformations are resolved, and the Cremona group acts by isometries on 
$\ZZ$  and on the hyperbolic space $\Hyp_\infty \subset \R^{1,\infty}$ (see~\cite{Cantat:Survey, Manin:book}):

\begin{thm}
Let $\bfk$ be an algebraically closed field. The  group $\Cr_2(\bfk)$ acts faithfully by linear isometries for the intersection form on $\ZZ \simeq \Z^{1,\infty}$. In particular, it 
acts faithfully by isometries on the infinite dimensional hyperbolic space $\Hyp_\infty\subset \R^{1,\infty}$.
\end{thm}

\begin{rem}
If $\bfk$ is not algebraically closed, we can choose an algebraic closure ${\overline{\bfk}}$ of ${\bfk}$, 
and embed $\Cr_2(\bfk)$ into $\Cr_2({\overline{\bfk}})$ to get a faithful action on an infinite dimensional 
hyperbolic space. 
\end{rem}
 
\section{Types of birational transformations}\label{par:IsomType}
 
From~Section~\ref{par:ZZ} we know that the Cremona group $\Cr_2(\bfk)$ acts faithfully by isometries 
on $\ZZ$ (with respect to the intersection product) and on $\Hyp_\infty$
(with respect to $\dist$). Since there are three types of isometries for such a space, 
we obtain three types of birational transformations of the plane: 
 elliptic, parabolic, and loxodromic elements of $\Cr_2(\bfk)$. 
We shall divide the proof of Theorem~A (and Theorem~A' below) in three cases, according to the type of $f\in \Cr_2(\bfk)$.

\subsection{Elliptic elements}

A birational transformation $f\in \Cr_2(\bfk)$ is {\bf{elliptic}} if and only if its orbits on $\Hyp_\infty$ 
are bounded, if and only if the sequence $(\deg(f^n))$ is bounded. In \cite{Blanc-Deserti}, J\'er\'emy Blanc
and Julie D\'eserti prove that any elliptic element $f$ of infinite order is conjugate to an element $f'$ of $\PGL_3(\bfk)$
 (their proof, written for $\bfk=\C$, works over
any algebraically closed field). We can now follow an argument of Marat Gizatullin (see \cite{Gizatullin_decompo_inertia}, Lemma~2) which is also given in the book \cite{Cerveau-Deserti} of Dominique Cerveau and Julie D\'eserti.
The smallest normal subgroup containing $f$ contains also $f'$, and because $\PGL_3(\bfk)$ is simple, it contains $\PGL_3(\bfk)$. In particular, it contains the involution $\eta ([x:y:z])=[-x:-y:z]$. But $\eta$ is conjugate to $\sigma ([x:y:z])=[yz:zx:xy]$ in $\Cr_2(\bfk)$ and by Noether-Castelnuovo theorem $\sigma$ 
and $\PGL_3(\bfk)$ generate $\Cr_2(\bfk)$; thus the normal 
subgroup generated by $f$ is equal to $\Cr_2(\bfk)$. This proves the 
 following lemma, and therefore also Theorem~A over arbitrary algebraically closed field for elliptic elements.
 
 \begin{lem}\label{lem:elliptic}
 Let $\bfk$ be an algebraically closed field. Let $f$ be a non-trivial element of $\PGL_3(\bfk)$ or
 an elliptic element of $\Cr_2(\bfk)$ of infinite order.
 The smallest normal subgroup of $\Cr_2(\bfk)$ containing $f$ is equal to $\Cr_2(\bfk)$.
 \end{lem}

\begin{rem}\label{rem:Zimmermann+Lamy}
Lemma \ref{lem:elliptic} requires $\bfk$ to be algebraically closed. When $\bfk$ is the field $\R$ of real numbers, the smallest normal subgroup generated by {\sl{any}} element of $\Cr_2(\R)$ is a proper normal subgroup of $\Cr_2(\R)$ (see \cite[Corollary 1.4]{Zimmermann:Duke}). 
The same result holds for perfect fields with at least one Galois extension of degree $8$, for instance $\Q$; this follows from \cite[Theorem C(2)]{Lamy-Zimmermann:JEMS}.
\end{rem}

\subsection{Parabolic elements and Jonqui\`eres twists}

By definition, a {\bf{parabolic}} element of $\Cr_2(\bfk)$ acts on $\Hyp_\infty$ without fixed point, but 
with a unique fixed point on the boundary $\partial \Hyp_\infty$. According to~\cite{Diller-Favre, Gizatullin:1980}, there are in fact
two types of parabolic elements: 
\begin{itemize}
\item[(1)] the sequence $(\deg(f^n))$ grows linearly, and in that case, one says that $f$ is a {\bf{Jonqui\`eres twist}}. 
\item[(2)] the sequence $(\deg(f^n))$ grows quadratically and in that case, $f$ is called a {\bf{Halphen twist}}. 
\end{itemize}
This result holds over any field.

Let us assume that $f$ is a Jonqui\`eres twist first (we shall study Halphen twists in Section~\ref{par:Halphen}), and 
that $\bfk$ is algebraically closed. 
Then, it is proved in \cite{Diller-Favre} that $f$ preserves a 
unique pencil of rational curves in $\P^2_\bfk$ (this result of Jeffrey Diller and Charles Favre is stated for the field of complex numbers, but the proofs apply to any algebraically closed field). In particular, $f$ is a {\bf{Jonqui\`eres transformation}}: it is a Cremona transformation that preserves a pencil of rational curves in $\P^2_\bfk$. Up to conjugacy by an element of $\Cr_2(\bfk)$, there is a unique pencil of rational curves in 
the plane, namely the pencil of lines through a point. Thus, $f$, as well as any Jonqui\`eres transformation, is conjugate
to a birational transformation of the affine plane $\bbA^2_\bfk$ that permutes the vertical lines;
\begin{equation}\label{eq:Jonq}
f(x,y)=(A(x), B_x(y))
\end{equation}
where $A$ is in $\PGL_2(\bfk)$ 
and $B_x$ is in $\PGL_2(\bfk(x))$. Marat Gizatullin proved in \cite[Lemma 2]{Gizatullin_decompo_inertia} that the smallest normal subgroup of $\Cr_2(\bfk)$ containing such an element $f\neq \id$ coincides always with $\Cr_2(\bfk)$ (see also \cite[Prop. 5.21]{Cerveau-Deserti}). (\footnote{The proof works as follows. Write $f$ as in Equation~\eqref{eq:Jonq}, and choose $C_x\in \PGL_2(\bfk(x))$ such that the transformation 
$h(x,y)= (x,C_x(y))$ does not commute to $f$. Then, the commutator $[f,h]$ is an element of the simple group $\PGL_2(\bfk(x))$. This proves that
the smallest normal subgroup containing $f$ contains $\PGL_2(\bfk)$, and then one can apply Lemma~\ref{lem:elliptic}.}) Thus, we obtain the following lemma.

 \begin{lem}\label{lem:JonquieresAlgClosed}
 Let $\bfk$ be an algebraically closed field. If $f\in \Cr_2(\bfk)$ is a Jonqui\`eres transformation (for instance a Jonqui\`eres twist), 
 the smallest normal subgroup of $\Cr_2(\bfk)$ containing $f$ coincides with $\Cr_2(\bfk)$.
 \end{lem}

This proves Theorem~A over any algebraically closed field 
when $f$ is a Jonqui\`eres twist, because all iterates $f^n$, $n\neq 0$, are again Jonqui\`eres twists.  

\begin{rem}
Lemma~\ref{lem:JonquieresAlgClosed} requires $\bfk$ be algebraically closed (see Remark~\ref{rem:Zimmermann+Lamy}). 
\end{rem}

\begin{rem}
Most elements of finite order in $\Cr_2(\bfk)$ are conjugate to Jonqui\`eres transformations; as such, they do not generate
a proper normal subgroup.  When $\bfk=\C$, we know that there are $29$ families of finite order elements which are not conjugate to Jonqui\`eres transformations, and each of them is conjugate to an automorphism of a del Pezzo surface of degree $1$, $2$ or $3$ (see \cite[Theorem 3]{Blanc_subgroup_finite_order}). It would be interesting to decide whether some of them may generate a proper normal subgroup of $\Cr_2(\bfk)$.
\end{rem}

\subsection{Loxodromic transformations}\label{par:Loxo}
An element $f$ of $\Cr_2(\bfk)$ is {\bf{loxodromic}} if it fixes two boundary points of $\Hyp_\infty$ 
and acts as a non-trivial translation on the geodesic joining these two points. If $L(f)$ is the
hyperbolic length of this translation, then the sequence $(\deg(f^n))$ grows like $\exp(\vert n\vert L(f))$
as $n$ goes to $\pm\infty$. By definition, the number 
\begin{equation}
\lambda(f)=\exp(L(f))=\lim_{n\to +\infty} \deg(f^n)^{1/n}
\end{equation}
is the {\bf{dynamical degree}} of $f$ (see~\cite{Cantat:Annals}).

\subsubsection{$p$-Automorphisms (see~\cite{Cantat:Survey, SB})}
Before studying all loxodromic elements, we focus on a class of birational transformations
that is defined only when the characteristic of $\bfk$ is positive. Note that Theorem~\ref{thm:dynamical-degree} below has been 
obtained under slightly more restrictive hypotheses by Nicholas Shepherd-Barron in a recent version of~\cite{SB}.

Assume that $\bfk$ is a field of characteristic $\fieldchar(\bfk)=p>0$.
Let $\Frob(t)=t^p$ be the Frobenius endomorphism of $\bfk$. An element $\varphi(t)$ of $\bfk[t]$  is a {\bf{linearized polynomial}} or is a {\bf{$p$-polynomial}} if all its monomials have degree ${p^j}$ for some $j\geq 0$; in other words, one can write 
\begin{equation}
\varphi(t)=\sum_j a_j t^{p^j}=\sum_j a_j \Frob^{j}(t)
\end{equation}
where $\Frob^{j}$ is the $j$-th iterate of the Frobenius endomorphism, i.e. $t^{p^j}=\Frob^j(t)=\Frob\circ \Frob \circ \cdots \circ \Frob(t)$ ($j$ compositions). 
These $p$-polynomials are exactly the polynomial transformations of $\bfk$ which are additive $\varphi(x+y)=\varphi(x)+\varphi(y)$. 
The composition of two $p$-polynomials $\varphi$ and $\psi\colon \bfk \to \bfk$ is another $p$-polynomial; with 
the laws given by addition and composition, the set of linearized polynomials  $\bfk[\Frob]$ is, naturally, a non-commutative $\bfk$-algebra. 

Let $\bbG_a$ denote the additive group of dimension $1$.
Every $2\times 2$ matrix with coefficients in $\bfk[\Frob]$  
determines an algebraic endomorphism of the algebraic group
$\bbG_a(\bfk)\times \bbG_a(\bfk)$: if $a$, $b$, $c$ and $d$ are the coefficients of the matrix, the endomorphism is
given by 
\begin{equation}\label{eq:lin}
f(x,y)=(a(x)+b(y),c(x)+d(y)).
\end{equation}
It is invertible if and only if the matrix is invertible over the non-commutative ring $\bfk[\Frob]$.

\begin{lem}\label{lem:add-auto}
Let $\bfk$ be a field of characteristic $p>0$. The group of algebraic automorphisms of the algebraic group $\bbG_a(\bfk)\times \bbG_a(\bfk)$
coincides with the group $\GL_2(\bfk[\Frob])$. 
\end{lem}

\begin{proof}
Consider an algebraic automorphism $F$ of the algebraic group $\bbG_a(\bfk)\times\bbG_a(\bfk)$.
Writing 
$F(x,y)=(U_1(x,y),U_2(x,y))$, we see that $U_i$  satisfies the relation $U_i(x_1+x_2,y_1+y_2)=U_i(x_1,y_1)+U_i(x_2,y_2)$. Looking at the highest degree
term $x^ky^\ell$ of $U_i$ (in the lexicographic order), one sees that it must be a power of $x$ or $y$, of degree $p^m$ for some $m$. 
Then, by induction on the degree, $U_1(x,y)=a(x)+b(y)$ and $U_2(x,y)=c(x)+d(y)$ for some pairs of linearized polynomials. Since $F$ is invertible, one
concludes that $F$ is an element of $\GL_2(\bfk[\Frob])$.
\end{proof}

 Consider  the  group $\Trans\subset \Cr_2(\bfk)$ consisting of all translations 
 \begin{equation}\label{eq:translation}
 t_{u,v}(x,y)=(x+u,y+v);
 \end{equation}
 this group is isomorphic to $\bbG_a(\bfk)\times \bbG_a(\bfk)$. 
We define the {\bf{group $\pAut(\bbA^2_\bfk)$ of $p$-automorphisms}} of the affine plane
as the normalizer of $\Trans$ in $\Cr_2(\bfk)$. 

\begin{lem}\label{lem:pAut} The group $\pAut(\bbA^2_\bfk)$ is a subgroup of $\Aut(\bbA^2_\bfk)$, and 
\[
\pAut(\bbA^2_\bfk)=   (\bbG_a(\bfk)\times\bbG_a(\bfk)) \rtimes\GL_2(\bfk[\Frob]).
\]
\end{lem}

\begin{proof} 
Let $f\in \Cr_2(\bfk)$ be an element that normalizes $\Trans$; here, $f$ is viewed as a birational transformation of $\bbA^2_\bfk$.
For every $(a,b)\in \bbG_a(\bfk)\times\bbG_a(\bfk)$, there exists a unique $(\tilde a,\tilde b)\in \bbG_a(\bfk)\times\bbG_a(\bfk)$
such that 
\begin{equation}\label{eq:f-normalizer}
f\circ t_{a,b}\circ f^{-1} =t_{\tilde a,\tilde b}.
\end{equation}
The map $(a,b)\mapsto (\tilde a,\tilde b)$ is the group-automorphism of  $ \bbG_a(\bfk)\times\bbG_a(\bfk)$ determined by conjugacy by $f$.
Equivalently, $f\circ t_{a,b}=t_{\tilde a,\tilde b}\circ f$. This equation shows that the set $\Ind(f)\subset \bbA^2_\bfk$ of indeterminacies of $f$ is invariant under 
translations, thus $\Ind(f)$ is empty, and $\Ind(f^{-1})$ too by the same reasoning. So, $f$ is a regular automorphism of $\bbG_a(\bfk)\times\bbG_a(\bfk)$.
Changing $f$ into $f\circ t_{u,v}$ with $(u,v)=f^{-1}(0,0)$, we assume that $f$ fixes the origin. Evaluating
Equation~\eqref{eq:f-normalizer} at the origin, we get $(\tilde{a}, \tilde{b}) = f(a,b)$. Thus, $f$ is now an algebraic automorphism 
of the group  $ \bbG_a(\bfk)\times\bbG_a(\bfk)$, and the conclusion follows from Lemma~\ref{lem:add-auto}.\end{proof}

When we say that an element $q$ of $\bfk[\Frob]$
has degree $d$ (with respect to $\Frob$), we mean that $q=\sum_{j=0}^d q_j \Frob^j$ with $q_d\neq 0$; 
writing $\Frob^j(t)=t^{p^j}$, $q$ becomes an element of $\bfk[t]$ of degree $p^d$. 

Let $\Aff(\bbA^2_\bfk)=(\bbG_a(\bfk)\times\bbG_a(\bfk)) \rtimes\GL_2(\bfk[\Frob])$ be the group of affine transformations of the affine plane $\bbA^2_\bfk$,
a subgroup of $\pAut(\bbA^2_\bfk)$.  
Let $\El(\bbA^2_\bfk)$ be the group of all elementary automorphisms,
 i.e. automorphisms \begin{equation}\label{eq:elementary}
h(x,y)=(ux+q(y),vy+w)
\end{equation}
with $q\in \bfk[y]$, $u$, $v$, and $w$ in $\bfk$, and $uv\neq 0$.
The subgroup $\Elp(\bbA^2_\bfk)=\El(\bbA^2_\bfk)\cap \pAut(\bbA^2_\bfk)$ 
of elementary $p$-automorphisms consists of automorphisms $h$ as in Equation~\eqref{eq:elementary}
with $q\in \bfk[\Frob]$.
\begin{thm}\label{thm:dynamical-degree}
Let $\bfk$ be a field of characteristic $p>0$. 
\begin{enumerate}
\item\label{thm_amalgamated_product} The group 
$\pAut(\bbA^2_\bfk)$
 is generated by $\Aff(\bbA^2_\bfk)$ and $\Elp(\bbA^2_\bfk)$. It is the amalgamated product of these two groups along their intersection. 

\item The dynamical degree of any   $p$-automorphism is a non-negative power of $p$. 
\end{enumerate}
\end{thm}

\begin{proof}
\newcommand{\pE}{\mathsf{pE}}\newcommand{\E}{\mathsf{E}}
\newcommand{\A}{\mathsf{A}}
To ligthen notations, we set $\A=\Aff(\bbA^2_\bfk)$, $\E=\El(\bbA^2_\bfk)$ and
$\pE=\Elp(\bbA^2_\bfk)$.
We first prove that $\A$ and $\pE$ generate $\pAut(\bbA^2_\bfk)$.
By Lemma \ref{lem:pAut}, it suffices to prove that $\GL_2(\bfk[\Frob])\subset \langle \A,\pE\rangle$.

Consider $f\in \GL_2(\bfk[\Frob])$ with coefficients $a$, $b$, $c$, and $d$ in $\bfk[\Frob]$, as in Equation~\eqref{eq:lin}.
If $c=0$, then $f(x,y)=(a(x)+b(y), d(y))$, and since $f$ is an automorphism of $\bbA^2_\bfk$ the degree of $a(t)$ and $d(t)$ with respect to $t$
must be equal to~$1$. Thus, $f$ is an elementary $p$-automorphism, i.e. $f\in \pE$. 
If $a=0$, consider the linear automorphism $J\in \A$ defined by $J(x,y)=(y,x)$.
Then as above, $J\circ f\in \pE$ and we are done. 

Assume now that $ac\neq 0$. 
Write $a=\sum_{i=0}^m a_i\Frob^j$ and $c=\sum_{j=0}^n c_j\Frob^j$ with $m=\deg(a)$ and $n=\deg(c)$ (the degrees are with respect to $\Frob$). 
We argue by induction on the complexity $\deg(a)+\deg(c)$.
Up to changing $f$ into $J\circ f$, we may assume that $m\leq n$. 
Set $u=c_n(\Frob^{n-m}(a_m))^{-1}$, and compose $f$ with the $p$-automorphism
\begin{equation}
g(x,y)= \left( x, y-u\Frob^{n-m}(x) \right) \in J\circ \pE\circ J^{-1}   
\end{equation}
to get a new element $g\circ f$ of $\GL_2(\bfk[\Frob])$ with lower complexity.
We conclude by induction that $g\circ f$ and therefore $f$ lie in $\langle \A,\pE\rangle$.
 
To conclude the proof of the first assertion, set ${\mathsf{S}}=\A\cap\pE=\A\cap \E$;  
this is the group of affine automorphisms of $\bbA^2_\bfk$ whose linear part is upper triangular. 
By the theorem of Jung and van der Kulk, one has a decomposition into an amalgamated product $\Aut(\bbA^2_\bfk)=\A*_{\mathsf{S}} \E$.
One can then conclude the proof of the first assertion using a reduced form argument in the amalgamated product, or arguing geometrically as follows.
From Theorem 7 of \cite{Serre:AASL2}, chapter 4,
$\Aut(\bbA^2_\bfk)$ acts on a tree $\calT$, with a fundamental domain given by an edge $e=v_1v_2$, where the stabilizer of $e$, $v_1$ and $v_2$ are respectively ${\mathsf{S}}$, $\A$ and $\E$.
Restricting this action to the subgroup $\pAut(\bbA^2_\bfk)$, 
the stabilizers of $e$, $v_1$ and $v_2$ in $\pAut(\bbA^2_\bfk)$ are respectively ${\mathsf{S}}$, $\A$ and $\pE$.
The orbit of $e$ under $\pAut(\bbA^2_\bfk)$ is a subtree of $\calT$, to which we can apply Theorem 6 of \cite{Serre:AASL2}, chapter 4: this gives
the amalgamated product structure we were looking for.  

For the second assertion, 
take a $p$-automorphism $f$, and write it as a composition 
\begin{equation}\label{eq_reduced_form}
f= g_\ell\circ g_{\ell-1} \circ   \cdots \circ g_1
\end{equation}
where each factor $g_i$ is an element of $\A$ or $\pE$.
We say that such a decomposition 
is reduced if none of the $g_i$ lies in ${\mathsf{S}}$,  
and no two consecutive $g_i$ lie in the same factor.
Every $p$-automorphism has such a reduced decomposition, unless it is an element of ${\mathsf{S}}$. 
We say that it is cyclically reduced if, moreover, $g_n$ and $g_1$ 
are not both in $\A$ or $\pE$. 

If $f$ lies in $\pE$ or $\A$,  then the degree of $f^n$ is bounded by $\deg(f)$, so the dynamical degree of $f$
is $\lambda(f)=1$.
If $f$ is written as a reduced composition of length $\geq 2$,
Theorem 2.1 in~\cite{Friedland-Milnor} asserts that $\deg(f)=\prod_i \deg(g_i)$
(\footnote{The proof, written over $\R$ or $\C$ in \cite{Friedland-Milnor}, applies verbatim over any field. 
The main remark is the following: start with an automorphism $h_0=(h_0^x(x,y), h_0^y(x,y))$ that satisfies $\deg(h_0^x)\leq \deg(h_0^y)=d_0$;
after composition with an element $h_1$ 
of $\pE\setminus \A$ of degree $d\geq 2$, we get an automorphism 
$h_1\circ h_0=(h'^x,h'^y)$ such that $\deg(h'^x)=dd_0> \deg(h'^y)$, and composing with an element $h_2$  of $\A\setminus \pE$ 
we obtain an automorphism $h_2\circ h_1\circ h_0=(h''^x,h''^y)$ such that $\deg(h''^x)\leq  \deg(h''^y)=dd_0$. Then, to prove Theorem~2.1 of~\cite{Friedland-Milnor}, do an induction on the length $\ell$ of $f$ in Equation~\eqref{eq_reduced_form}, starting with $g_1$ if it is in $\A$, or with $g_0=\id$ otherwise.
}).
Conjugating $f$ in $\pAut(\bbA^2_\bfk)$, we can assume that this decomposition is cyclically reduced 
so that the composition 
\begin{equation}
f^n= (g_\ell\circ g_{\ell-1} \circ   \cdots \circ g_1)\circ \cdots \circ  (g_\ell\circ g_{\ell-1} \circ   \cdots \circ g_1)
\end{equation}
is also reduced for all $n\geq 1$, so
 $\deg(f^n)=\prod_i \deg(g_i)^n=\deg(f)^n$. So $\lambda(f)=\deg(f)$, and this number is 
 a power of $p$ because $f$ is a $p$-automorphism. 
\end{proof}

\subsubsection{A result of Shepherd-Barron}
In~\cite{Cantat-Lamy}, a criterion is given to show that (a large iterate of) a given loxodromic element $f$ of the Cremona 
group $\Cr_2(\bfk)$ generates a proper normal subgroup of $\Cr_2(\bfk)$. In~\cite{SB}, Shepherd-Barron 
proves that this criterion is satisfied by every loxodromic element $f$, except in two cases: when $f$ is conjugate to a monomial transformation, 
or when $\fieldchar(\bfk)=p>0$ and $f$ is conjugate to a  $p$-automorphism of the plane. These results prove one
direction of the following theorem.

\begin{thm-C}
Let $\bfk$ be an algebraically closed field, and let $f$ be a loxodromic element of $\Cr_2(\bfk)$. The following properties
are equivalent:
\begin{itemize}
\item[(a)] there is an integer $\ell \geq 1$ such that $f^{\ell}$ generates a  proper normal subgroup of $\Cr_2(\bfk)$; 
\item[(b)] $f$ is not conjugate to a monomial transformation or, when $p:=\fieldchar(\bfk)>0$, to a $p$-automorphism of the plane. 
\end{itemize}
\end{thm-C}

To conclude the proof of this theorem, one needs to prove that monomial transformations and $p$-automorphisms
do not generate proper normal subgroups. Let us do it when $f\neq \id$ is a $p$-automorphism. Write $f(x,y)=F(x,y)\circ t_{a,b}$
where $t_{a,b}$ is a translation. Since $f$ is loxodromic, $F$ is not the identity. Then, for 
every translation $t_{u,v}$ (as in Equation~\ref{eq:translation}), one gets
\begin{equation}
f\circ t_{u,v}\circ f^{-1}\circ t_{u,v}^{-1}= t_{u',v'}
\end{equation}
where $(u'+u,v'+v)=F(u,v)$. One can choose $t_{u,v}$ in such a way that $t_{u',v'}\neq \id$, because $F$ is not the identity. 
This implies that the smallest normal subgroup containing $f$ contains a non-trivial element $t_{u',v'}$ 
of $\PGL_3(\bfk)$, and the conclusion follows from Lemma~\ref{lem:elliptic}. The proof is the same for monomial 
transformations, replacing the translations $t_{u,v}$ by the diagonal transformations $d_{u,v}\colon(x,y)\mapsto (ux, vy)$ with $uv\neq 0$
(i.e. the additive group $\bbG_a\times \bbG_a$ by the multiplicative group $\bbG_m\times \bbG_m$). 

\begin{rem}
When the field is not algebraically closed, the theorem of Shepherd-Barron may be stated as follows: let $f$ be a loxodromic element of
$\Cr_2(\bfk)$; if the normal subgroup generated by $f^n$ co\"incides with $\Cr_2(\bfk)$ for all $n\neq 0$, then $f$ normalizes a commutative
algebraic subgroup $G\subset \Cr_2(\bfk)$ such that $G(\bfk)$ is infinite and $G$ becomes isomorphic to $\bbG_a^2$ or $\bbG_m^2$ on any
algebraic closure of $\bfk$. In $\Cr_2(\R)$, the derived subgroup is a subgroup of infinite index that contains all monomial transformations 
(this follows from \cite[Theorem 1.1]{Zimmermann:Duke} because the group of monomial transformations is generated by elements of
$\Bir(\P^2_\bfk)$ of degree $\leq 2$).
If $\bfk$ is any perfect field with an algebraic Galois extension of degree $8$ (for instance any number field), St\'ephane Lamy and Susanna Zimmermann 
show that the smallest normal subgroup of $\Cr_2(\bfk)$ containing any given monomial transformation is a proper subgroup of $\Cr_2(\bfk)$ (\cite{Lamy-Zimmermann:JEMS}).
 \end{rem}

\subsubsection{An example}

Let $m$, and $q$ be positive integers. Set $n=m^2-1$. Consider the element $h$ of $\Cr_2(\bfk)$ defined in affine coordinates by
\begin{equation}
h(x,y)=(y,x+y^e)
\end{equation}
with $e=m$ modulo $n$.
Let $\alpha$ be a root of unity such that $\alpha^{n}=1$ (i.e. $\alpha^{m^2}=\alpha$). Set $\beta=\alpha^m$ and consider the diagonal 
linear automorphism $g(x,y)=(\alpha x, \beta y)$.
Then 
\begin{equation}
h\circ g \circ h^{-1}= g^m.
\end{equation}
Now, compose $h$ with a monomial transformation $f(x,y)=(x^ay^b,x^cy^d)$. Then, 
$(f\circ h)\circ g \circ (f\circ h)^{-1}=g^q$ if and only if the integers $a$, $b$, $c$, and $d$
satisfy 
\begin{equation}\label{eq:congruences}
am+b=q \quad {\mathrm{and}} \quad cm+d=qm \quad \mod (n).
\end{equation}
We can choose $c=q$ and $d=0$ modulo $n$, and then $b$ such that $bq=\pm 1$ modulo $n$ (this is possible
if $q$ is invertible modulo $n$). Since $m$ is invertible modulo $n$, we then set $a=(q-b)/m$ modulo $n$.
Thus, the system of congruences \eqref{eq:congruences} has  
solutions for which $ad-bc=\pm 1$  modulo $n$;
then, one can lift these solutions to elements of $\GL_2(\Z)$. One gets more solutions by composition with elements
in the kernel of the projection $\GL_2(\Z)\to \GL_2(\Z/n\Z)$.

\begin{pro}
Let $\bfk$ be a field. Let  $m$ and $q$ be positive integers,  set $n=m^2-1$, and assume that $q$ is invertible modulo $n$.
Then, there is a  loxodromic element $s\in \Cr_2(\bfk)$ and an elliptic element $g\in \PGL_3(\bfk)\subset \Cr_2(\bfk)$
such that
\begin{enumerate}
\item $g$ is of order $n$;
\item $s \circ g \circ s^{-1}=g^q$; 
\item $s$ is not conjugate to a monomial map or a $p$-automorphism (if $\fieldchar(\bfk)=p>0$).
\end{enumerate}
\end{pro}
\begin{proof}[Sketch of proof]
Take $f$, $g$, and $h$ as above, and set $s=f\circ h$. All we need to show is that one can choose $f$ in such 
a way that $s$ is not conjugate to a monomial map or a $p$-automorphism. Composing $f$ with monomial 
maps given by matrices $M\in \GL_2(\Z)$ which are equal to ${\mathrm{Id}}$ modulo $n$, and choosing $e$ large, we may assume
that the entries of $f$ satisfy $e > c > d > a > b >1$. Then, a simple recursion shows that the degree of $(f\circ h)^k$
is equal to $(c+de)^k$; in particular, the dynamical degree of $f\circ h$ is equal to the integer $c+de$. Thus, $f\circ h$
is not conjugate to a monomial map because the dynamical degree of a loxodromic monomial map is not an integer: it is
a quadratic unit. And, changing $e$ if necessary, one sees that $f\circ h$ is not conjugate to a $p$-automorphism because the dynamical degree of such 
an automorphism of the plane is a power of $p$ (see Theorem \ref{thm:dynamical-degree}). 
\end{proof}

Pick such a pair of elements $(s,g)$ in the Cremona group. The smallest normal subgroup generated by 
$s^k$ contains 
\begin{equation}
[s^k,g]=s^k\circ g \circ s^{-k}\circ g^{-1}=g^{q^k-1}.
\end{equation}
Thus, if $q^k\neq 1$ modulo $n$ for all $1\leq k\leq \ell$, the smallest normal subgroup generated by $s^k$, for any $k\leq \ell$, contains 
a non-trivial elliptic element and  co\"{\i}ncides with $\Cr_2(\bfk)$ if $\bfk$ is algebraically closed (Lemma~\ref{lem:elliptic}). We obtain: 

\begin{thm}\label{thm:non-uniform}
Let $\bfk$ be an algebraically closed field. For 
every integer $\ell$, there is a loxodromic element $s$ in $\Cr_2(\bfk)$ such that  $s^n$ generates a proper
normal subgroup of $\Cr_2(\bfk)$ for some $n>\ell$, but not for $n\leq \ell$. \end{thm}

This result shows that, unlike in Theorem B, one cannot take $n$ independent of $f$ in Theorem A and Theorem A' below.

\subsection{Theorem A'}\label{par:charpos}

We can now state the extension of Theorem~A' to algebraically closed fields of arbitrary characteristic. 

\medskip

\begin{thm-AA}
Let $\bfk$ be an algebraically closed field.
Let $f$ be an element of $\Cr_2(\bfk)$ of infinite order. The following properties are equivalent. 
\begin{itemize}
\item[(a)] There exists an integer $n>0$ such that $f^n$ generates a non-trivial proper normal subgroup of $\Cr_2(\bfk)$. 
\item[(b)] The birational transformation $f$ is  a Halphen twist or it is a loxodromic element that is not conjugate to a monomial transformation
or  to a $p$-automorphism when $p:=\fieldchar(\bfk)>0$.
\end{itemize}
\end{thm-AA}
\medskip

We already proved this result for elements $f\in \Cr_2(\bfk)$ which are not Halphen twists. The rest of this paper is 
devoted to this last case.  

\section{Halphen twists, Halphen pencils, and translation lengths}\label{par:Halphen}

By definition, a Halphen twist $f$ in $\Cr_2(\bfk)$ is a birational transformation $f\colon \P^2_\bfk\to \P^2_\bfk$
such that $\deg(f^n)$ grows quadratically.
In this section, we describe a result of Marat Gizatullin showing that
such a birational transformation preserves a unique pencil of curves of genus $1$. But before that, we start with 
a description of all such pencils and their geometry. 

\subsection{Horoballs}\label{par:horoballs_cr}
As in Section~\ref{par:Z1n}, consider the space $\R^{1,n}$, together with its lattice $\Z^{1,n}$ and 
hyperbolic subset $\Hyp_n$.
Assume that the vector $\xi=a_0\bfe_0-\sum_j a_j\bfe_j$ of $\R^{1,n}$ satisfies $a_0>0$, 
and $\xi \cdot \xi=a_0^2-\sum_j a_j^2=0$. Then, $\xi$ determines a boundary point
of the hyperbolic space $\Hyp_n$. Let $\epsilon$ be a positive real number. 
The {\bf{horoball}} $H_\xi(\epsilon)$ in $\Hyp_n$ is the subset
\[
H_\xi(\epsilon)=\{v\in \Hyp_n  \; ; \; 0 < v\cdot \xi  \leq \epsilon\}.
\]
It is a limit of balls with centers converging to the boundary point $\xi$. The horosphere is
the boundary 
\begin{equation}
\partial H_\xi(\epsilon)=\{v\in \Hyp_n  \; ; \;  v\cdot \xi = \epsilon\}.
\end{equation}
\begin{rem}
If $h$ is any isometry of $\Hyp_n$, then 
\begin{equation}
h(H_\xi(\epsilon))=H_{h(\xi)}(\epsilon).
\end{equation}
If $h$ fixes the boundary point $\R\xi\in \partial\Hyp_n$ then $h$ maps $H_\xi(\epsilon)$ to $H_\xi(e^{\pm L(h)} \epsilon)$, where $L(h)$ 
is the translation length of $h$ and the sign is positive (resp. negative) when $\xi$ is attracting (resp. repelling).
\end{rem}

 \subsection{The lattice $\Z^{1,9}$ and the group $W_9$} 
 
 Consider the lattice $\Z^{1,n}$ introduced in Subsection~\ref{par:Z1n}, but now for the specific dimension $n=9$. 
The \emph{anti-canonical vector} $\xi_9=3\bfe_0-\sum_{j=1}^9 \bfe_i$ is
isotropic, and the ray $\R \xi_9$ determines a boundary 
 point of the hyperbolic space $\Hyp_9$. The horoballs centered at this boundary point are the subsets 
$H_{\xi_9}(\epsilon)=\{u\in \Hyp_9\; ; \; 0 < u\cdot \xi_9 \leq \epsilon \}.$
The riemannian metric induced on $\partial H_{\xi_9}(\epsilon)$ by the hyperbolic metric is euclidean, and we can identify  
$\partial H_{\xi_9}(\epsilon)$ to $\R^8$ with a euclidean metric $\dist_{euc}$ (see \cite[Part II, Lemma 11.32]{BH}).
The euclidean and hyperbolic distances satisfy
  \begin{equation}
\frac 12\dist_{euc}(x,y)=\sinh( \frac12 \dist(x,y))\label{eq_eucl}
\end{equation}
for all pairs of points $x$ and $y$ on $\partial H_{\xi_9}(\epsilon)$
(see \cite{Smogorzhevsky}[\S13] for instance).

\begin{rem}The horospheres  $\partial H_{\xi_9}(\epsilon)$ foliate $\Hyp_9$ and $\partial H_{\xi_9}(3)$ is the leaf containing $\bfe_0$. 
The geodesics of $\Hyp_9$ with one end point equal to $\xi_9$ form a transverse foliation;
if we follow  this geodesic foliation, we get a family of (holonomy) maps $h_\epsilon\colon \partial H_{\xi_9}(3)\to \partial H_{\xi_9}(\epsilon)$.
In the half-space model with $\xi_9$ at infinity, the horospheres are horizontal hyperplanes and the $h_\epsilon$ are vertical translations. 
They are not isometric, but are similitudes with respect to the euclidean metrics on the horospheres: $h_\epsilon$ multiplies 
euclidean distances by $\epsilon/3$.\end{rem}

In the orthogonal group $\O(\Z^{1,9})$, 
there is an index $2$ subgroup $\O(\Z^{1,9})^+$ preserving $\Hyp_9$, and thus acting by isometries on $\Hyp_9$. The subgroup 
$\O(\Z^{1,9}; \xi_9)^+$ fixing the isotropic line $\R\xi_9$ fixes also the class $\xi_9$, because it is the unique primitive
integral vector on this line with $a_0>0$; this group preserves every horosphere $\partial H_{\xi_9}(\epsilon)$, acting on it as a group of affine isometries 
for the euclidean metric $\dist_{euc}$. The action of $\O(\Z^{1,9}; \xi_9)^+$ on $\partial H_{\xi_9}(3)$ is conjugate to the action on $\partial H_{\xi_9}(\epsilon)$
by the similitude $h_\epsilon$. 

The group of affine isometries of the euclidean space  $\partial H_{\xi_9}(\epsilon)$
is an extension of an orthogonal group by its group of translations.
Since the group $\O(\Z^{1,9}; \xi_9)^+$ is a discrete subgroup of this group of isometries,
Bieberbach theorem shows that there is a finite index subgroup $P^*$ in $\O(\Z^{1,9}; \xi_9)^+$ acting by translations on the horosphere $\partial H_{\xi_9}(3)$.
For $g$ in $P^*$, the translation length on $\partial  H_{\xi_9}(3)$ with respect to $\dist_{euc}$ is equal to $\dist_{euc}(\bfe_0,g(\bfe_0))$.
And if $g\in P^*\setminus\{\id\}$, $\langle \bfe_0,g(\bfe_0)\rangle$ is an integer $\geq 2$, so the hyperbolic distance satisfies $\dist(\bfe_0,g(\bfe_0))\geq \arcosh(2)$; in view of Equation~\eqref{eq_eucl}, the euclidean distance is bounded by
\begin{equation}
\dist_{euc}(\bfe_0,g(\bfe_0))\geq 2\sinh\left( \frac12 \arcosh(2) \right)=\sqrt{2}.
\end{equation}
Since the conjugacy $h_\epsilon$ is a similitude that multiplies distances by $\epsilon/3$, 
the group $P^*$ acts also by 
translations on each of the  horospheres  $\partial H_{\xi_9}(\epsilon)$, with euclidean translation length bounded from 
below by 
\begin{equation}
\Delta_{euc}(\epsilon):=
\frac{\eps\sqrt{2}}{3}.
\end{equation}
In terms of the hyperbolic metric, this says that for all $x\in \partial H_{\xi_9}(\epsilon)$ and all  $g\in P^*\setminus\{\id\}$ one has
\begin{equation}\label{eq_Delta}
\dist(x,gx)\geq \Delta(\epsilon):=2\arcsinh\left( \frac{\eps}{3\sqrt{2}}\right).
\end{equation}

Denoting by $P^*_n$ the subgroup of $P^*$ consisting of all $n$-th powers of elements of $P^*$,
one gets the following result. 

\begin{lem}\label{lem:P}
There is a finite index, normal subgroup $P^*\subset  \O(\Z^{1,9}; \xi_9)^+$ with the following properties. For every $\epsilon >0$, there is 
a positive real number $\Delta(\epsilon)$, such that
\begin{enumerate}
\item the group $P^*$ is isomorphic to the abelian group $\Z^8$;
\item it acts by translations on the euclidean horosphere $\partial H_{\xi_9}(\epsilon)$; 
\item the translation length of any element of $P^*_n\setminus \{\id\}$ on $\partial H_{\xi_9}(\epsilon)$ for the euclidean distance is at least $n \Delta_{euc}(\epsilon)=n\epsilon\sqrt{2}/3$
 and for all $x\in\partial H_{\xi_9}(\epsilon)$,
    $\dist(x,gx)\geq \Delta(n\epsilon)=2\arcsinh\left(n\frac{\eps}{3\sqrt{2}}\right).$
\end{enumerate}
\end{lem}

\begin{rem}[See~\cite{Cantat-Dolgachev}]
One can, in fact, be much more precise. 
The orthogonal complement of the vector
$\xi_9$ 
is a sublattice $\bfE_9\subset \Z^{1,9}$. A basis of $\bfE_9$ is formed by the vectors
$\alpha_0 =\bfe_0 -\bfe_1 -\bfe_2 -\bfe_3$  and $\alpha_i =\bfe_i -\bfe_{i+1}$, for $i=1$, $\ldots$, $8$.
The intersection matrix  $(\alpha_i\cdot\alpha_j )$ is equal to $\Gamma_9 - 2\id_9$ , where $\Gamma_9$ is the incidence matrix of the 
Dynkin diagram of type $E_9$. 
In particular, each class $\alpha_i$ has self- intersection $-2$ and determines an involutive isometry of $\Z^{1,9}$, namely
$s_i \colon x\mapsto x+(x\cdot \alpha_i)\alpha_i$. By definition, these involutions generate the Weyl (or Coxeter) group $W_9$, and one shows that $W_9$ has finite index
in $\O(\Z^{1,9}; \xi_9)^+$.
 When restricted to $\bfE_9$, the intersection form is degenerate; its radical is generated by the vector $\xi_9$ and the lattice $\bfE_8 \simeq \bfE_9/\Z \xi_9$
  is isomorphic to the root lattice of finite type $E_8$. Then, $W_9$ is isomorphic to the affine Weyl group of type $E_8$, and fits in the extension
$0\to \bfE_8\to W_9\to W_8\to 1 $ where the injection $\iota \colon \bfE_8 \to W_9$ of the additive group $\bfE_8\simeq \Z^8$ into $W_9$ is defined by the following formula: for $w\in \bfE_8$,
\begin{equation}
 \iota(w):v\in \Z^{1,9}\mapsto v - (v\cdot \xi_9)w + \left( (w\cdot v) - \frac{1}{2} (v\cdot  \xi_9) (w\cdot w)\right)\xi_9.
 \end{equation}
 The quotient $W_8$ is a finite reflection group. In Lemma~\ref{lem:P}, one can take $P^*=\iota(\bfE_8)$.
 \end{rem}

\subsection{Halphen pencils and Halphen surfaces (see \cite{Dolgachev:Halphen, Cantat-Dolgachev})}\label{par:P(X)}

A {\bf{Halphen pencil}} ${\mathfrak{p}}$ of index $m$ is an irreducible  pencil of plane curves of degree $3m$ with $9$ base-points of multiplicity $\geq m$; here, infinitely near points may be included in the list of base-points. 
In particular, a Halphen pencil of index $1$ is made of cubic curves with no fixed component. 
A   smooth rational surface $X$ is a {\bf{Halphen surface}} if there exists an integer $m > 0$ such that the linear system $\vert - mK_X\vert$ is of dimension $1$, has no fixed component, and has no base-point (here, $K_X$ denotes the canonical divisor). 
The index of a Halphen surface is the smallest possible value for such a positive integer $m$.

As shown in \cite{Cantat-Dolgachev}, every Halphen surface $X$ is obtained by blowing-up the nine base-points of a Halphen pencil ${\mathfrak{p}}$.
In particular, the Picard number of $X$ is equal to $10$: there is a basis of the N\'eron-Severi group given by the class $\bfe_0$ of a line 
in $\P^2$, and the classes $\bfe_j$ of the $9$ exceptional divisors (coming from the $9$ base-points). 
The pencil ${\mathfrak{p}}$ determines a 
fibration $\pi\colon X\to \P^1$; this fibration is the same as the one given by the linear system $\vert - mK_X\vert$.
Two cases may appear: if the general fiber of $\pi$ is a smooth curve of genus~$1$ we say that $\pi$ is elliptic; otherwise, the
general fiber is a rational curve with a cusp and one says that $\pi$ is quasi-elliptic. Quasi-elliptic examples occur only in characteristic
$2$ and $3$ (see the first sections of~\cite{Bombieri-Mumford} or the last chapter of~\cite{Cossec-Dolgachev}); the cusps of the fibers form a smooth curve which will be denoted by $\Sigma_X\subset X$ (computing the intersection of   $\Sigma_X$ with a fiber, one sees that
$m$ divides $p$, so $m\leq 3$).

When $m>1$, the pencil contains a unique cubic curve with multiplicity $m$. 
The class of this curve in the Picard-Manin space is the class of $-K_X$ and it 
is equal to $\xi_X=3\bfe_0-\bfe_1-\ldots - \bfe_9$; if we want to refer to the pencil, rather
than to the surface, we shall denote this class by $\xi_{\mathfrak{p}}$ (it may be considered as a point of $\NS(X)$ or of $\ZZ$). Thus, $\xi_X$ corresponds to the anti-canonical vector $\xi_9$ under
the natural isometry $\Z^{1,9}\to \NS(X)$.
The model $\Z^{1,9}$ does not depend on the pencil ${\mathfrak{p}}$ (or on the Halphen surface $X$). 

The group of automorphisms $\Aut(X)$ acts on $\NS(X)$ and its image is a group of isometries fixing the anti-canonical class $\xi_X$. 
Since $\Aut(X)$ preserves the canonical bundle, it permutes the sections of $\vert-mK_X\vert$, and it preserves the fibration $\pi$ (permuting its
fibers). 

Consider the group $\Bir(\P_\bfk^2; {\mathfrak{p}})$
of birational transformations of the plane preserving the pencil ${\mathfrak{p}}$. 
 After conjugacy by the blow-up $X\to \P^2_\bfk$,  $\Bir(\P_\bfk^2; {\mathfrak{p}})$ becomes a subgroup of $\Bir(X)$ that
 permutes the fibers of $\pi$. The fibration $\pi$ is relatively minimal, which amounts to say that there
 is no exceptional divisor of the first kind in the fibers of $\pi$ (see~\cite{Grivaux}); this implies that $\Bir(\P_\bfk^2; {\mathfrak{p}})$ is contained in $\Aut(X)$ (see~\cite{Iskovskikh-Shafarevich}). 
Since $\Aut(X)$ permutes the fibers of $\pi$, we conclude that $\Bir(\P_\bfk^2; {\mathfrak{p}})=\Aut(X)$ (after conjugacy by the blow-up $X\to \P^2$): 

\begin{lem}\label{lem:Halphen-bir-aut-fib}
Let ${\mathfrak{p}}$ be a Halphen pencil. Let $X$ be the Halphen surface and $\pi\colon X\to \P^1_\bfk$ be the fibration defined by ${\mathfrak{p}}$.
Then $\Aut(X)$ permutes the fibers of $\pi$, and is conjugate to $\Bir(\P_\bfk^2; {\mathfrak{p}})$ by the blow-up of the nine
base-points of ${\mathfrak{p}}$.
\end{lem}

Let $P_0(X)\subset \Aut(X)$ (resp. $P_0({\mathfrak{p}})\subset \Bir(\P^2_\bfk, {\mathfrak{p}})$) be the subgroup of automorphisms 
$f$ such that the action $f^*\colon \NS(X)\to \NS(X)$ corresponds to an element of the abelian group $P^*\subset \O(\Z^{1,9})^+$ given 
by Lemma~\ref{lem:P}. 

In the following lemma, $\xi_X$ is viewed as a boundary point of $\Hyp_\infty$, via the natural embedding 
$\NS(X)\to \ZZ$,  the horosphere 
$\partial H_{\xi_X}(\epsilon)$ is the horosphere in $\Hyp_\infty$ centered at $\xi_X$, and the real number $\Delta(\epsilon)$ is defined in Equation \eqref{eq_Delta}.

\begin{lem}\label{lem:PX}
The group $P_0({\mathfrak{p}})$ is a normal subgroup of $\Bir(\P^2_\bfk; {\mathfrak{p}})$. For every $\epsilon >0$, every non-trivial element $f$ of this group 
acts
\begin{itemize}
\item on $\partial H_{\xi_X}(\epsilon)\cap  \Hyp_X$ as a euclidean translation, of length $\geq \Delta_{euc}(\epsilon)$;
\item on $\partial H_{\xi_X}(\epsilon)\subset \Hyp_\infty$ as a fixed point free isometry that satisfies 
\[
\dist(u,f(u))\geq \Delta(\epsilon)
\]
 for every $u\in \partial H_{\xi_X}(\epsilon)$.
\end{itemize}
Moreover, for each $l>0$, there exists an integer $n\geq 1$ that does not depend on the pencil ${\mathfrak{p}}$, 
such that 
\[ \dist_{\Hyp_\infty}(x,g^n (x))\geq l
\]
for each $g\in P_0({\mathfrak{p}})\setminus\{1\}$
and each point $x$ in the horosphere $\partial H_{\xi_X}(\epsilon)$, .
\end{lem}

\begin{proof} 
The first property  follows directly from the definition of $P_0({\mathfrak{p}})$. For the second one, 
we use that $\partial H_{\xi_X}(\epsilon)\cap  \Hyp_X$ is a $\Bir(\P^2_\bfk; {\mathfrak{p}})$-invariant euclidean subspace of the (infinite dimensional) euclidean space $\partial H_{\xi_X}(\epsilon)$,
so the euclidean orthogonal projection commutes with $\Bir(\P^2_\bfk; {\mathfrak{p}})$.
Since it is $1$-Lipschitz, it follows that $\dist_{euc}(u,f(u))\geq \Delta_{euc}(\eps)$ for all $u\in\partial H_{\xi_X}(\epsilon)$.
By Equation~\eqref{eq_eucl}, the hyperbolic distance satisfies $\dist(u,f(u))\geq \Delta(\eps)$.
The ``moreover part'' is now a consequence of Lemma \ref{lem:P} and $\Delta(n\eps)\to \infty$ as $n\to \infty$.
\end{proof}

\subsection{Gizatullin's theorem}
The following result encapsulates a theorem of Marat Gizatullin (see~\cite{Gizatullin:1980}) and results of Georges H. Halphen and Igor Dolgachev (see~\cite{Dolgachev:Halphen}). 

\begin{thm}\label{thm:giza}
There is an integer $B_H\leq 8!$ with the following property. 
Let $\bfk$ be an algebraically closed field. Let $f\in \Cr_2(\bfk)$ be a Halphen twist.
Then, up to conjugacy, $f$ preserves a unique Halphen pencil ${\mathfrak{p}}$ (of some index $m\geq 1$), and:
\begin{itemize}
\item this 
Halphen pencil provides an elliptic fibration of the corresponding Halphen surface $X$, 
\item the kernel of the homomorphism $\Aut(X)\to \GL(\NS(X))$ is finite, of order $\leq B_H$.
\end{itemize}
\end{thm}
In other words, one can conjugate $f$ by a birational transformation $g\colon \P^2_\bfk\dasharrow \P^2_\bfk$
and find a Halphen pencil ${\mathfrak{p}}$, of some index $m$, such that $f':=g\circ f \circ g^{-1}$ is contained
in $\Bir(\P^2_\bfk; {\mathfrak{p}})$. Blowing-up the base-points of ${\mathfrak{p}}$, $f'$ becomes an automorphism
of a rational surface $X$, preserving a relatively minimal fibration $\pi:X\to \P^1_\bfk$ of genus $1$. 
The class of a general fiber is equal to $m\xi_X$, where $\xi_X=\xi_{\mathfrak{p}}=3\bfe_0-\bfe_1-\ldots -\bfe_9$. 
Viewed in $\Hyp_\infty$, this class $\xi_{\mathfrak{p}}$ gives the unique boundary point fixed by $f'$. Conversely, if $h$ is an element of $\Cr_2(\bfk)$ that fixes $\xi_{\mathfrak{p}}$, then $h$ determines an automorphism of $X$ preserving
the fibration $\pi$; then, either $h$ is elliptic, or it is a Halphen twist. If $h$ is a Halphen twist, then some iterate of $h$ is in the group $P_0(X)$. 

\begin{eg}
Consider, over the field of complex numbers, an elliptic curve $E=\C/\Lambda$, and form the product $\P^1_\C\times E$.
The involution $\sigma \colon ([x:y],z)\mapsto ([-x:y],-z)$ of $\P^1_\C\times E$ preserves exactly two fibers $\{[0:1]\}\times E$
and $\{[1:0]\}\times E$, with $4$ fixed points on each of them. Blowing up these $8$ points, and then taking the quotient by (the lift of) 
$\sigma$, we get a rational surface $X_E$, with a genus~$1$ fibration $\pi\colon X_E\to \P^1_\C$, and exactly two singular fibers
(each of type ${\rm{I}}_0^*$ in Kodaira's table). The multiplicative group acts by $([x:y],z)\mapsto ([sx:y],z)$ on $\P^1_\C\times E$, 
hence also on $X_E$. So, the kernel of the homomorphism $\Aut(X_E)\mapsto \GL(\NS(X_E))$ contains a multiplicative group; in particular 
it is infinite. According to Theorem~\ref{thm:giza}, there is no Halphen twist preserving the fibration $\pi\colon X_E\to \P^1_\C$. This can be seen 
directly on that example: a finite index subgroup of $\Aut(X_E)$ preserves the $8$ exceptional curves of self-intersection $-2$ contained in 
the fibers above $[0:1]$ and $[1:0]$, as well as the class of a general fiber; as a consequence, this finite index subgroup acts trivially on the 
subspace $\xi_9^\perp\subset \Z^{1,9}$, hence also on $\Z^{1,9}$. So, $\Aut(X_E)$ cannot contain any Halphen twist.\end{eg}

\begin{proof}  The first assertion, i.e. the existence of an invariant Halphen pencil ${\mathfrak{p}}$ (up to conjugacy) is due
to Gizatullin (see~\cite{Gizatullin:1980}): his proof works also in characteristic $2$ or $3$, as noted in \cite{Cantat-Dolgachev}.
Blowing-up the base-points of ${\mathfrak{p}}$, one gets a Halphen surface $X$, and $f$ is conjugate to 
an automorphism $f_X$ of $X$; this automorphism preserves the fibration $\pi\colon X\to \P^1$ given by ${\mathfrak{p}}$.

If $\pi$ were quasi-elliptic, $f_X$ would preserve the cusp curve $\Sigma_X$. Thus, $f_X$ would preserve the class $[\Sigma_X]$ of the curve of cusps, as well as the class $\xi_X$ of a general fiber. Since the intersection product $[\Sigma_X]\cdot \xi_X$ is positive, 
$f_X$ should preserve a class $[\Sigma_X]+m\xi_X$ with positive self-intersection, and $f_X$ should be elliptic, in contradiction with $f$ being a Halphen twist. Thus, $\pi$ is an elliptic fibration: its general fibers are curves of genus $1$.
 
Consider the kernel $K$ of the homomorphism $\Aut(X)\to \GL(\NS(X))$. Since ${\mathsf{Pic}}^0(X)$ is trivial,  $K$ 
is a linear algebraic group; we denote by $K^0$ the connected component of the identity in $K$. First, we show that $K$ is finite, or equivalently that $K^0$ is trivial. 
Recall from Lemma~\ref{lem:Halphen-bir-aut-fib} that $\Aut(X)$ permutes the fibers of $\pi\colon X\to \P^1_\bfk$: this
provides an algebraic homomorphism $\rho\colon K\to \Aut(\P^1)$ such that $\pi\circ g=\rho(g)\circ \pi$ for all $g\in K$. The kernel of
$\rho$ is a linear algebraic group, preserving each of the fibers of $\pi$. Since every algebraic homomorphism of a connected linear algebraic
group to a curve of genus $1$ is trivial, the kernel of $\rho$ is finite. The Euler characteristic of $X$ (using $\ell$-adic cohomology if $\fieldchar(\bfk)>0$) is 
equal to $12$ because $X$ is obtained from $\P^2_\bfk$ by a sequence of $9$ 
successive blow-ups; hence, $\pi$ has at least one singular fiber (see~\cite{Cossec-Dolgachev}, Proposition 5.1.6, p. 290). Since, the group $\rho(K^0)\subset \PGL_2(\bfk)$ fixes all the 
critical values of $\pi$, it fixes at least one point; in particular, $\dim(K^0)=0$, $1$, or $2$.  If $\dim(K^0)=1$, the (Zariski closures) of the orbits of $K^0$ form a pencil of rational curves, and this 
pencil is $\Aut(X)$-invariant, because $K^0$ is a normal subgroup of $\Aut(X)$. Then, $f_X$ would preserve two pencils,
the Halphen pencil and this pencil of rational curves,
and $f$ would not be a Halphen twist: the two pencils would provide two fixed points of $f$ on the
boundary of $\Hyp_\infty$, and $f$ would not be a parabolic isometry of $\Hyp_\infty$. 
This contradiction shows that $\dim(K^0)\neq 1$.
If the dimension of $K^0$ is $2$, then $\rho(K^0)$ is the group of affine
transformations $z\mapsto az+b$ of the line;  the subgroup of elements $g$ in $K^0$ such that $\rho(g)$ is a translation $z\mapsto z+b$ is invariant under 
the action of $\Aut(X)$ on $K^0$ by conjugacy, its dimension is equal to $1$, and we also get a contradiction. Thus, $K$ is finite.

Now, we show that $\vert K\vert\leq 8!$. Consider the birational morphism $\tau\colon X\to \P^2_\bfk$ given by (the inverse) of the blow-up
of the base-points of the pencil ${\mathfrak{p}}$. Since $K$ acts trivially on $\NS(X)$, it preserves the classes of the $9$  irreducible curves contracted by $\tau$, and since these curves have negative self-intersections it preserves each of these curves. Thus, there is a finite subgroup $K_0$ of $\Aut(\P^2_\bfk)$
and an isomorphism $\varphi\colon K\to K_0$ such that $\varphi(g)\circ \tau=\tau\circ g$ for every $g\in K$. The group $K_0$ preserves 
$\mathfrak{p}$, permuting its members but fixing each of its base-points (including infinitely near ones). Conversely, the subgroup $G$
of $\Aut(\P^2_\bfk)$ that preserves each of these $9$ points is an algebraic subgroup that lifts to a subgroup of $\Aut(X)$, and this subgroup acts trivially on $\NS(X)$
because it preserves the vectors $\bfe_1,\dots, \bfe_9$ of the corresponding geometric basis as well as the vector $\bfe_0$ (since $G$ preserves the class of a line in $\P^2_\bfk$). 
Thus, $G$   coincides with $K_0$ and is finite. So, we just have to prove that $\vert G\vert \leq 8!$.
For that, we identify $\Aut(\P^2_\bfk)=\PGL_3(\bfk)$ to the open subset of $\P^8_\bfk$ corresponding to the coefficients $[h_{i,j}]$ of matrices
such that $\det(h_{i,j})\neq 0$. The finite group $G$ is determined by $9$ algebraic equations $h(q_k)=q_k$, where the $q_k$ are 
the base-points of ${\mathfrak{p}}$. Let us show that the degrees of these equations is bounded by $8$. Indeed, if $h=(h_{i,j})$ is an element of $\PGL_3(\bfk)$ and $q_0$ is a point of $\P^2_\bfk$, then the equation $h(q_0)=q_0$ is linear in the
coefficients $h_{i,j}$. Then, if we blow-up $q_0$ and choose a point $q_1$ on the exceptional divisor, the equation $h(q_1)=q_1$ is linear
in the coefficients of $Dh_{q_0}$, and therefore also in the coefficients of $h$. Going on one step further, blowinq-up $q_1$ and choosing a point 
$q_2$ on the exceptional divisor, the equation $h(q_2)=q_2$ is a linear constraint on the second jet of $h$ at $q_0$; thus, it is a quadratic equation in the coefficients $[h_{i,j}]$ (\footnote{ To see it, assume $h$ fixes the point $[0:0:1]$, then in the affine coordinates $(x,y)$ we have 
$$
h\left(  \begin{array}{c} x\\ y \end{array} \right)=\left(  \begin{array}{c} \frac{a x+by}{1+ux+vy}\\ \frac{cx+dy}{1+ux+vy} \end{array} \right)
=\left(\begin{array}{c}
(ax+by)(1-(ux+vy)+(ux+vy)^2+\cdots)\\
(cx+dy)(1-(ux+vy)+(ux+vy)^2+\cdots)
\end{array}\right).
$$
Thus, the coefficients of the jet of degree $d$ of $h$ at the origin $(0,0)$ are polynomial functions of degree $\leq d$ in the coefficients of $h$.
}). If we do a tower of $k$ successive blow-ups, with $q_{i+1}$ infinitely near $q_i$, then the equation $h(q_k)=q_k$ has degree 
$\max(k-1,1)$ in the coefficients of $h$. Since, in our setting, we have $9$ base-points, we get $9$ equations of respective degrees $d_1$, $\ldots$, $d_9$ 
in the coefficients $[h_{i,j}]$.  By the Bezout Theorem, the number of solutions of this system of equations, if finite, is at most $d_1\times \cdots \times d_9$. The worst case is for 
\begin{equation}
(d_1,d_2,d_3,\ldots, d_9)=(1,1,2,\ldots, 8)
\end{equation}
and gives the upper bound  $8!=40320$. 
\end{proof}

\begin{eg} Keep the notation of the proof. Given any Halphen pencil, the group $K$ can be described explicitly. We just give
two examples. First, assume that the $9$ base-points of the pencil ${\mathfrak{p}}$ contain a 
projective basis $(q_0, q_1, q_2,q_3)$ of $\P^2_\bfk$. Then $K$ is trivial. This corresponds to the generic situation. 

Now, assume that the index $m$ of $\mathfrak{p}$ is larger than $1$ and the unique multiple fiber $C$ of
the fibration $\pi\colon X\to \P^1_\bfk$ is smooth. Then, the projection $C_0=\tau(C)\subset \P^2_\bfk$ is a smooth cubic
curve, and is invariant under $K_0$. Since $K_0$ fixes at least one point of $C_0$, it acts by group-automorphisms on $C_0$. Moreover, 
 an element of $\Aut(\P^2_\bfk)$ which preserves $C_0$ pointwise is the identity. So, $\vert K_0\vert=\vert K\vert$ is bounded by $6$ 
 if $\fieldchar(\bfk)\neq 2,3$, by $12$ if $\fieldchar(\bfk)=3$, and by $24$ if $\fieldchar(\bfk)=2$. 
 \end{eg}

\begin{rem}
It may be the case that $K$ is automatically trivial when  $X$ supports a Halphen twist. We
don't have any counter-example to this assertion.
\end{rem}

Recall that $P_0({\mathfrak{p}})\subset \Bir(\P^2_\bfk; {\mathfrak{p}})$ is the subgroup of automorphisms 
$f\in \Bir(\P^2_\bfk; {\mathfrak{p}})$ whose image in $\GL(\NS(X))$ 
lies in the abelian group $P^*\subset \O(\Z^{1,9})^+$ given 
by Lemma~\ref{lem:P}. 

\begin{cor}\label{cor:P}
Let $f$ be a Halphen twist, and let ${\mathfrak{p}}$ be the unique $f$-invariant Halphen pencil. 
The subgroup $P_0({\mathfrak{p}})\subset \Bir(\P^2_\bfk;{\mathfrak{p}})$ is an extension 
\[
1\to K\to P_0({\mathfrak{p}})\to \Z^s\to 0
\] 
of an abelian group of rank $s \leq 8$ by a finite group $K$ with $\vert K\vert \leq B_H\leq 8!$.

There is a normal subgroup  $P({\mathfrak{p}})\triangleleft P_0(\mathfrak{p})$ of index $\leq s^{2\vert K\vert} ((\vert K\vert+1) !)$ which is a free abelian group of rank $s$. 
\end{cor}

This subgroup $P({\mathfrak{p}})$ of $\Bir(\P^2_\bfk;{\mathfrak{p}})$ will play an important role in the next sections.

\begin{proof}
The previous theorem shows the first assertion.
To prove the second assertion, consider any exact sequence $1\to K \to P_0\to \Z^s \to 0$ where $K$ is a finite group, and $\vert K\vert \leq D$ for some 
positive integer $D$. The group $P_0$ normalizes $K$, acting by conjugation on it. There is a subgroup 
$P_1$ of $P_0$ of index $\leq D!$ that centralizes $K$: every element of $P_1$ commutes to every element of $K$.
Set $n=2\vert K\vert$; hence $\vert K\vert$ divides $n$ and $n(n+1)/2$.
Then, set $P=\{ f^n\vert\;  f\in P_1\}$. The first remark is that 
$P$ is a subgroup of $P_1$. Indeed, given any pair of elements $f$ and $g$ in $P_1$, we have 
$fg=gfc$ for some $c$ in $K$; then, for every $m\geq 1$, $f^mg=gf^mc^m$ and $fg^m=g^mfc^m$ because $c$ is central, 
and this gives $(fg)^m=g^mf^mc^{1+2+3+\ldots +m}$; with $m=n$ we obtain that $f^ng^n=(fg)^n$ so that 
$P$ is stable under multiplication. The second remark is that $P\cap K=\{1\}$. The
third remark is that the projection of $P$ in $\Z^s$ is an injective homomorphism, the image of which is a finite index subgroup of  $n\Z^s$. Thus, $P$ is a free abelian subgroup of $P_1$ of rank $s$ and index $\leq s^nD$. Since $n=2\vert K\vert$
we see that the index of $P$ in $P_0$ is $\leq s^{2D} D (D!)\leq s^{2D} ((D+1)!)$. Moreover, $P$ is a normal subgroup of $P_0$ (because $P_1$
is normal and $P$ is characteristic). 
\end{proof}

\section{Disjonction of horoballs}\label{par:disjonction}

In this section, we prove Corollary~\ref{prop:horo_R_separated} which says that horoballs
associated to Halphen twists are disjoint. This is the first technical input to apply results of 
small cancellation and geometric
group theory. 

Consider a family $\mathcal{C}$ of convex subsets of $\Hyp_\infty$. Let $R$ be a positive real number. By definition,  $\mathcal{C}$  is {\bf $R$-separated} if $\dist_{\Hyp_\infty}(C,C')> R$
for all pairs of distinct elements $C$, $C'\in\mathcal{C}$.
A vector $v\in \Z^{1,\infty}$ is {\bf{primitive}} if it is not a non-trivial multiple of some vector $v'\in\Z^{1,\infty}$.

\begin{thm}\label{thm:disjonc}
Let $\mathcal{P}$ be the set of integral, isotropic and primitive vectors $v$ in $\Z^{1,\infty}$ 
such that $v \cdot \bfe_0 >0$.
If $0<\epsilon<1/\sqrt{2}$, the family of horoballs $\{H_v(\epsilon)\}_{v\in\mathcal{P}}$ is $\arcosh(\frac{1}{4\epsilon^2})$-separated. 
\end{thm}

\begin{rem}
The set of horoballs $\{H_v(\epsilon)\}_{v\in\mathcal{P}}$ is invariant by the group of isometries of $\Hyp_\infty$ preserving integral points. 
 \end{rem}

\begin{proof}
	Consider two different horoballs $H_{v}(\epsilon)$ and $H_{v'}(\epsilon)$ of $\mathcal{P}$. 
	Consider $x\in [v,v']\cap \partial H_{v}(\epsilon)$ and $y\in[v,v']\cap \partial H_{v'}(\epsilon) $, where $[v,v']$ is the geodesic from $v$ to $v'$.
	Set $m=v\cdot v'$ (this is a positive integer), and write $x = \alpha v + \alpha' v'$ and $y= \beta v + \beta ' v'$.
	The points $x$ and $y$ belong to $\Hyp_\infty$ so $x\cdot x = y\cdot y = 1$, and by definition of the horospheres  $x\cdot v = y\cdot v'=\epsilon$. 
	This gives
	\begin{align}
	& 2\alpha\alpha 'm = 2\beta\beta 'm =1,\\
	& \alpha ' m = \beta m =\epsilon.
	\end{align}
	Using these equalities we get:
	\begin{equation}
         x\cdot y= (\alpha\beta'+\alpha'\beta)m = \frac{m}{4\epsilon^2} +\frac{\epsilon^2}{m}.
         \end{equation}
	Since $m\geq 1$, we obtain:
	\begin{equation}
        x\cdot y > \frac{1}{4\epsilon ^2}. 
        \end{equation}
	For $4\epsilon^2 > 1$ the two horoballs overlap, for $4\epsilon^2 =1$ they are tangent at $x=y$, and for $4\epsilon^2 < 1$
	they are disjoint: The distance between $H_{v}(\epsilon)$ and $H_{v'}(\epsilon)$ is  the distance between $x$ and $y$ and is equal to
	\begin{equation}
	\dist_{\Hyp_\infty}\left(H_{v}(\epsilon),H_{v'}(\epsilon)\right)=\dist_{\Hyp_\infty}(x,y)\geq\arcosh(\frac{1}{4\epsilon^2}).
	\end{equation}
	This concludes the proof.
	\end{proof}
 

Let $\mathfrak{p}$ be a Halphen pencil.
Recall from section~Section \ref{par:P(X)} that  $\xi_{\mathfrak{p}}$ denotes the
anticanonical class of the surface obtained by blowing up the nine base-points of $\mathfrak{p}$. 
Since $\xi_{\mathfrak{p}}\in \ZZ$ is primitive and isotropic, and the Cremona group permutes the set of primitive, isotropic vectors, 
Theorem \ref{thm:disjonc} provides the following result.

\begin{cor}\label{prop:horo_R_separated}
Consider the set  of all horoballs of the form
\[
f(H_{\xi_{\mathfrak{p}}}(\epsilon))=H_{f(\xi_{\mathfrak{p}})}(\epsilon),
\]
where ${\mathfrak{p}}$ is any Halphen pencil and $f$ is any element of $\Cr_2(\bfk)$.
Then for $0< \epsilon < 1/\sqrt{2}$, this set of horoballs is $\arcosh(\frac{1}{4\epsilon^2})$-separated. 
\end{cor}

\section{Geometric group theory and conclusion}

To prove Theorem~B, we rely on the work of Fran\c{c}ois Dahmani, Vincent Guirardel and Denis Osin (see \cite{DGO}).

\subsection{Hyperbolic spaces and rotating families}\label{par:DGO}
First we recall one of the main results of \cite{DGO}.

Consider a geodesic metric space $X$ and a non-negative constant $\sigma$. A subset $Q$ of $X$ is {\bf $\sigma$-strongly quasiconvex} if for any two points $x,y$ of $Q$ there exist $x',y'$ in $Q$ and geodesics $[x',y']$, $[x',x]$ and $[y,y']$ included in $Q$ such that $\max\{\dist(x,x'),\dist(y,y')\}\leq \sigma$. 

Let $\delta$ be a non-negative constant. 
A geodesic metric space $X$ is   {\bf $\delta$-hyperbolic} if every triangle in $X$ is $\delta$-thin, 
meaning that each side of the triangle is included in the $\delta$-neighborhood of the union of the two remaining sides.
We shall say that $X$ is {\bf Gromov-hyperbolic} if it is $\delta$-hyperbolic for some $\delta >0$.


Let $G$ be a group acting by isometries on a $\delta$-hyperbolic space $X$, for some $\delta>0$. Consider a $G$-invariant subset $C\subset X$, which we call the set of {\bf apices}. For each $c\in C$, consider $G_c\subset \Stab(c)$ a subgroup of the stabilizer of $c$ in $G$, called the {\bf rotation subgroup}. 
Following \cite[Definition 5.1]{DGO}, we say that
the family $(C,\{G_c\}_{c\in C})$ is a {\bf very rotating family} if it satisfies the following conditions:
\begin{enumerate}
	\item for every $g\in G$, for every $c\in C$, $G_{gc}=gG_cg^{-1}$,
	\item for every $c\in C$, for every $g\in G_c\setminus\{1\}$, and for every $x$, $y\in X$ satisfying 
	\[
	\quad \quad \dist(x,c)\in [20\delta,40\delta], \, \; \dist(y,c)\in [20\delta,40\delta], \; \, {\text{and}} \; \dist(gx,y)\leq 15 \delta,
	\]
	any geodesic between $x$ and $y$ contains $c$. 
\end{enumerate}
Moreover, as in Section~\ref{par:disjonction}, the family is $\rho$-separated, for some $\rho>0$, if the distance between distinct apices is always strictly bigger than $\rho$. 

\begin{thm}[{\cite[Theorem 5.3]{DGO}}]\label{thm:normal_subgroup}
Let $G$ be a group acting by isometries on a $\delta$-hyperbolic geodesic space $X$. Let $(C,\{G_c\}_{c\in C})$ be a $\rho$-separated very rotating family, for some $\rho> 200\delta$. Then the normal subgroup $N_G$ of $G$ generated by the groups $\{G_c\}_{c\in C}$ is a free product of a (usually infinite) subfamily of $\{G_c\}_{c\in C }$. Moreover, for every $c\in C$, $N_G\cap \Stab(c)=G_c$.
\end{thm}

\begin{rem}
If $G_c$ is a proper subgroup of $\Stab(c)$, or if $G_c$ is abelian and $G$ has no 
non-trivial homomorphism onto an abelian group, then the normal subgroup generated by the groups $\{G_c\}_{c\in C}$
is a proper subgroup of $G$. In particular, if $f\in G_c\setminus\{\id\}$, the smallest normal subgroup of $G$ containing $f$ 
 is proper and non-trivial. See {\cite{DGO}} for other consequences.
\end{rem}

\subsection{Cone-off construction}

In our case, we  shall apply Theorem~\ref{thm:normal_subgroup} to $G=\Cr_2(\bfk)$, and the role of $G_c$ will be played by the subgroup $P(\mathfrak{p})$ of $\Bir(\P^2_\bfk, \mathfrak{p})$.
This group fixes a point $\xi_{\mathfrak{p}}\in \partial\Hyp_\infty$, but has no fixed point inside $\Hyp_\infty$. We shall modify $\Hyp_\infty$ 
to bring this fixed point inside the space: this is done in Section~\ref{subsec:cone-off}. Then, we need to modify this cone-off, because if we go deeply inside
the horoballs the very rotating property (2) above will fail. For that, we dig a hole in the cones, and this is explained in Section~\ref{subsec:parabolic-cone-off}.
For more details on this cone-off constructions, see \cite[Sections 5.3 and 7.1]{DGO} or \cite[Section 3]{CoAsp}. 

\subsubsection{Hyperbolic cones (see~\cite{BH})}
Consider  a metric space $Y$ and a real number $r_0>0$. The {\bf hyperbolic cone} over $Y$ of height $r_0$, denoted by $\Cone(Y,r_0)$, is the quotient of $Y\times [0,r_0]$ by the equivalence relation: 
\begin{equation}
(y_1,r_1) \sim (y_2,r_2) \; \text{ if and only if } \; r_1=r_2=0 \; \text{ or }  \; (y_1,r_1)=(y_2,r_2). 
\end{equation}
 By definition, the point $(y,0)$ is the {\bf apex} of $\Cone(Y,r_0)$. The metric on this cone is defined by the formula
\[ \dist_{\Cone(Y,r_0)}\left( (y,r),(y',r')\right)=\arcosh \left( \cosh r\cosh r'- \sinh r\sinh r' \cos \theta(y,y')  \right)    ,  \]
where $\theta(y,y')=\min \left( \pi,\frac{\dist_{Y}(y,y')}{\sinh r_0}\right)$. 
By \cite[Chap. I.5 Prop 5.10]{BH}, $\Cone(Y,r_0)$ is a geodesic metric space if and only if $Y$ is geodesic.

\subsubsection{Cone-off}\label{subsec:cone-off}

Consider a metric space $X$, a family $Y=\{Y_i\}_i$  of subsets of $X$, and $r_0$ a positive constant. 
The {\bf cone-off} of  $Y$ over $X$, denoted by $C(X,Y,r_0)$, is the space obtained from the disjoint union of $X$ and of the cones $\Cone(Y_i,r_0)$, for all $i$, by gluing each subspace $Y_i$ in $X$ to $Y_i\times \{r_0\}$ in $\Cone(Y_i,r_0)$ via the identification 
\begin{equation}
y\in Y_i\mapsto (y,r_0)\in \Cone(Y_i,r_0).
\end{equation}
By \cite[Proposition 2.1.5]{CoAsp}, for $(y,r_0)$ and $(y',r_0)\in \Cone(Y_i,r_0)$ we have
\begin{equation}
\dist_{\Cone(Y_i,r_0)}\big((y,r_0),(y',r_0)\big) \leq\dist_X(y,y') .  
\end{equation}
 Let us construct a metric on $C(X,Y,r_0)$.
Given $x$ and $x'$ in $C(X,Y,r_0)$, set
\[ \lVert x-x' \rVert =\begin{cases}
\dist_X(x,x') & \text{ if }x,x'\in X \text{ and there is no } i \text{ such that }x,x'\in Y_i\\
\dist_{\Cone(Y_i,r_0)}(x,x') & \text{ if there exists }Y_i\text{ such that }x,x'\in \Cone(Y_i,r_0)\\
+\infty &\text{ otherwise.}
\end{cases} \]
Then, consider a chain $C$ between $x$ and $x'$, i.e. a finite sequence of points $x=z_1$, $z_2$, $\dots$, $z_n=x'$ in $C(X,Y,r_0)$. Define its length to be
$\ell(C)=\sum_{i=1}^{n-1} \lVert z_{i+1}-z_i\rVert. $
By \cite[Proposition 3.1.7]{CoAsp}, a distance is defined on the cone-off as follows:
\begin{equation}
\dist_{C(X,Y,r_0)}=\inf \{\ell(C)\mid C \text{ is a chain between }x \text{ and } x'  \}.
\end{equation}


The following theorem concerns metric graphs with isometric edges (edges of the same length); 
this ensures that the cone-off is geodesic (\cite[Theorem I.7.19]{BH}). 
As in \cite{DGO}, page 77, set
\begin{equation}\label{eq:deltaU-rU}
\delta_U=900, \; {\text{and}}\;  r_U= 5\times 10^{12}.
\end{equation}

\begin{thm}\label{thm:hyperbolic_cone-off}
Given any $r_0\geq r_U$, there exists $\delta_c>0$ such that :
if $X$ is a $\delta_c$-hyperbolic metric graph with pairwise isometric edges, and if $\mathcal{Q}$ is a $40\delta_c$-separated family of $10\delta_c$-strongly quasiconvex subsets of $X$, then the cone-off $C(X,\mathcal{Q},r_0)$ is globally $\delta_U$-hyperbolic. 
\end{thm}

\begin{proof}
This statement follows from Corollary 5.39 of \cite{DGO}. Indeed, with the notation from \cite{DGO}, 
 the fellow traveling constant of the family  $\mathcal{Q}$ satisfies $\Delta(\mathcal{Q})=0$ because 
 $\mathcal{Q}$ is $40\delta_c$-separated. Hence, all hypotheses of Theorem~5.38 of \cite{DGO} are satisfied.
\end{proof}

\subsubsection{Parabolic cone-off (see~\cite[\S 7.1]{DGO})}\label{subsec:parabolic-cone-off}
For any $Q\in\mathcal{Q}$, we consider  the hyperbolic cone $\Cone(Q,r_0)$. For any $p$, $q\in \partial Q$ with 
$\dist_X(p,q)\leq \pi \sinh(r_0)$,
and for any geodesic $[p,q]$ in $\Cone(Q,r_0)$, we consider the filled triangle $T_{[p,q]}\subset \Cone(Q,r_0)$ bounded by the geodesics $[p,q]$, $[p,c_Q]$ and $[q,c_Q]$; this triangle is a cone over $[p,q]$.
For any $Q\in \mathcal{Q}$, we define 
\begin{equation}
B_{Q}={\bigcup} T_{[p,q]}
\end{equation}
where the union runs over all geodesics $[p,q]$ in $\Cone(Q,r_0)$ of length $\leq \pi \sinh(r_0)$ and with endpoints in $\partial Q$. 
\begin{defi}
Consider $X$ a Gromov-hyperbolic space and $\mathcal{Q}$ a $R$-separated family of convex subsets. The {\bf parabolic cone-off} $C'(X,\mathcal{Q},r_0)$ is the subset of $C(X,\mathcal{Q},r_0)$ given by:\[  \left(C(X,\mathcal{Q},r_0)\setminus \underset{Q\in\mathcal{Q}}{\bigcup}\Cone(Q,r_0) \right) \cup \left(\underset{Q\in\mathcal{Q}}{\bigcup} B_Q\right) .\]
It is endowed  with the induced path metric. 
\end{defi}

In the following lemma, $r_0\geq r_U$ and $\delta_c$ is given by Theorem \ref{thm:hyperbolic_cone-off}. For the definition of horoballs in a Gromov-hyperbolic graph see \cite[Section 7.1 p.111]{DGO}.

 \begin{lem}{\cite[Lemma 7.4]{DGO}}\label{lem:hyperbolicite-parabolic-coneoff}
Let $X$ be a $\delta_c$-hyperbolic graph with isometric 
edges and $\mathcal{H}$ a $40\delta_c$-separated system of horoballs. Then the parabolic cone-off $C'(X,\mathcal{Q},r_0)$ is $16\delta_U$-hyperbolic. 
 \end{lem}
 
The following Lemma is proved in \cite{DGO}. Although the statement is given in the context of relatively hyperbolic group,
its proof just needs a separated set of horoballs.

 \begin{lem}{\cite[Lemma 7.5]{DGO}}\label{lem:very_rotating}
Let $G$ be a group acting isometrically on a $\delta_c$-hyperbolic graph $X$ with isometric edges.
 Let $H_i\subset X$, $i\in I$, be a family of horoballs indexed by a set $I$. For each index $i$, 
 let $G(H_i)\subset G$ be the stabilizer of $H_i$ and $N_i\triangleleft G(H)$ be a normal subgroup of $G(H_i)$.
Assume that 
 \begin{itemize}
 \item[(i)] the $H_i$ are in pairwise disjoint orbits under the action of $G$;
 \item[(ii)] the set of distinct horoballs  \[\mathcal{H}=\{ g(H_i)\; \vert \; i\in I, \; g\in G/G(H_i)\}\] is $40\delta_c$-separated.
 \item[(iii)] $\dist_X(x,gx) \geq 4\pi \sinh(r_0)$
for each $g \in N_i\setminus\{1\}$ and each $x\in\partial H_i$.  
 \end{itemize}
Then the family of all $G$-conjugates of the $N_i$ defines a $2r_0$-separated very rotating family on $C'(X,\mathcal{H},r_0)$. 
 \end{lem}

\subsection{Conclusion}\label{par:conclusion}
We can now prove a strong form of Theorem B.

\begin{thm-BB}
There exists a positive integer $n_0$ with the following property. Let $\bfk$ be a field. 
Consider the collection of all iterates $g^{n_0}$ of all  Halphen twists $g\in \Cr_2(\bfk)$. 
Then the normal subgroup generated by this collection is a non-trivial proper subgroup of $\Cr_2(\bfk)$.

Moreover, this normal subgroup is a free product of free abelian groups of rank $\leq 8$.
\end{thm-BB}

\begin{proof}[Proof of Theorem B'] 	
First, assume that the field $\bfk$ is algebraically closed. 

	Consider the action of $ \Cr_2(\bfk)$ on $\Hyp_\infty$.  
		In order to apply the results of \cite{DGO}, which are stated in the case of  Gromov-hyperbolic graphs, we need  the  following classical construction. Let $\mathcal{G}$ be the metric graph defined as follows. The vertices of $\mathcal{G}$ are the points of $\Hyp_\infty$ and we put an edge between two vertices $x$ and $y$ if their distance in $\Hyp_\infty$ is at most~$1$. This graph is endowed with the path metric for which each edge is isometric to the interval $[0,1]$. The Cremona group acts on $\mathcal{G}$ by isometries. We denote by $i : \Hyp_\infty  \hookrightarrow \mathcal{G}$ the inclusion. 
		It is a $(1,1)$-quasi-isometry:
\begin{equation} \dist_{\mathcal{G}}(i(x),i(y))-1\leq \dist_{\Hyp_\infty}(x,y)\leq \dist_{\mathcal{G}}(i(x),i(y))
\end{equation}
for any $x,y\in \Hyp_\infty$. Then $\mathcal{G}$ is $\delta_{\mathcal{G}}$-hyperbolic with $\delta_{\mathcal{G}}=\ln(1+\sqrt{2})+1$.
	
Fix $\frac{1}{\sqrt{2}}>\epsilon>0$ such that $\arcosh(\frac{1}{4\epsilon^2})\geq 40\delta_{\mathcal{G}}$. 
For every Halphen pencil $\mathfrak{p}$, consider the horoball $H_{\mathfrak{p}}(\epsilon)$ in $\Hyp_\infty$. By Corollary \ref{prop:horo_R_separated}, the family of translates of all the $H_{\mathfrak{p}}(\epsilon)$ under $\Cr_2(\bfk)$ is $R_\eps$-separated for $R_\eps=\arcosh(\frac{1}{4\epsilon^2})\geq40\delta_{\mathcal{G}}$. To each horoball $H_{\mathfrak{p}}(\epsilon)$ of $\Hyp_\infty$, we associate the subgraph $H_{\mathfrak{p}}$ of $\mathcal{G}$ spanned by the vertices $\{i(x)\mid x\in H_{\mathfrak{p}}(\epsilon)\}$. 
This is a $0$-strongly quasiconvex subgraph of $\mathcal{G}$ and the family of all translates $gH_{\mathfrak{p}}$ (for $g$ in $\Cr_2(\bfk)$ and 
${\mathfrak{p}}$ a Halphen pencil) is again $R_\eps$-separated.
	
Set $r_0= r_U$. With $\lambda = \delta_c/\delta_{\mathcal{G}}$, the scaled graph $\lambda\mathcal{G}$ (in which each edge has length $\lambda$) 
is $\delta_c$-hyperbolic. 
We call it $\mathcal{G'}$, and we denote by $H'_\mathfrak{p}\subset \calg'$ and $\partial H'_\mathfrak{p}\subset\calg'$ the subsets of 
$\calg'$ corresponding to $i(H_\mathfrak{p}),i(\partial H_\mathfrak{p})\subset \calg$. Now the 
set $\mathcal{H'}$ of translates of $H'_{\mathfrak{p}}$ under $\Cr_2(\bfk)$, for all Halphen pencils invariant by some Halphen twists,
is $\lambda R_\eps$-separated; the choice for  $\epsilon$ implies that it is $40\delta_c$-separated.
By Lemma \ref{lem:hyperbolicite-parabolic-coneoff}, the parabolic cone-off $C'(\mathcal{G'},\mathcal{H'},r_0)$   
is $16\delta_U$-hyperbolic.

Let $\Bir(\P^2_\bfk;\mathfrak{p})$ be the stabilizer of $\mathfrak{p}$ in $\Cr_2(\bfk)$,
and let $P(\mathfrak{p})\triangleleft P$ be the normal free abelian subgroup given by Corollary \ref{cor:P}. 
Using Lemma \ref{lem:PX}, we find an integer
$n$, that does not depend on ${\mathfrak{p}}$, such that 
\begin{equation}
\dist_{\mathcal{G'}}(x,g^{n} x)\geq 4\pi\sinh(r_0)
\end{equation}
for any $g\in P(\mathfrak{p}) \setminus \{\id\}$ and for any $x\in \partial H_{\mathfrak{p}}\an{'}(\epsilon)$.
Then, we set $N_{\mathfrak{p}}=\{g^n | \; g \in P(\mathfrak{p})\}$. This is a subgroup because $P(\mathfrak{p})$ is abelian, and it is clearly normal in $\Bir(\P^2_\bfk;\mathfrak{p})$. Moreover, the Corollary \ref{cor:P} shows that there is a uniform integer
\begin{equation}
n_0=n\times \left( 8^{2B_H} \times (B_H+1)! \right)!
\end{equation} 
 such that $g^{n_0}$
is in some $N_{\mathfrak{p}}$ for every Halphen twist $g$.

The family of horoballs $H_{\mathfrak{p}}'$ and normal subgroups $N_{\mathfrak{p}}\subset \Bir(\P^2_\bfk;\mathfrak{p})$ 
satisfies  the hypotheses of Lemma \ref{lem:very_rotating}, so the family of $\Cr_2(\bfk)$-conjugates of the $N_{\mathfrak{p}}$ 
defines a $2r_0$-separated very rotating family on $C'(\mathcal{G'},\mathcal{H'},r_0)$. By hypothesis on $r_0=r_U=10^6\delta_U$, $2r_0>200\times (16\delta_U)$
(see Equation~\eqref{eq:deltaU-rU}). Applying Theorem \ref{thm:normal_subgroup}, we get the expected result. 

If the field $\bfk$ is not algebraically closed, the first step is to replace it by an algebraic closure $\bfk'$, so as to apply Theorem~\ref{thm:normal_subgroup}. 
The normal subgroup $F$ of $\Cr_2(\bfk')$ generated by the iterates $g^{n_0}$ for all Halphen twists $g\in \Cr_2(\bfk)$ is a free product of abelian groups. 
Since $\PGL_3(\bfk)$ does not embed into such a group, the intersection $F\cap \Cr_2(\bfk)$ is a proper normal subgroup of $\Cr_2(\bfk)$.

To conclude, we need to show that $\Cr_2(\bfk)$ always contains Halphen twists: this is done in the examples below. 
\end{proof}

\subsection{Three examples}\label{par:3examples}

\subsubsection{Halphen twists} Let $\bfk$ be a field of characteristic $p\neq 2$. 
 Let $C\subset \P^2_\bfk$ be a smooth cubic curve defined over $\bfk$, together
 with a point $q\in C(\bfk)$. By definition, the Jonqui\`eres involution $s_q$ is a birational involution of the plane that 
 preserves the pencil of lines through the point $q$ and fixes $C$ pointwise. If $L$ is a line containing $q$, 
 it intersects $C$ on a pair of points (except if the line is tangent to $C$) and the restriction $s_{q\vert L}$ 
 is the unique involution fixing these two points (this involution does not exist if $\fieldchar(\bfk)=2$).
 This involution is defined over $\bfk$.
 If one blows up the point $q$ and the  points of tangency of the $4$ lines containing $q$ which are tangent to $C$, 
 the involution $s_q$ lifts to a regular involution of a rational surface $X_q\to \P^2_\bfk$. 
 
 Now, if we start with
 two points $q$ and $q'$ in $C(\bfk)$, we can lift $s_q$ and $s_{q'}$ simultaneously as a pair of automorphisms
 $s_Y$ and $s'_Y$ 
 of a rational surface $Y$ ($Y$ is the blow-up of the plane in $10$ points that dominates $X_q$ and $X_{q'}$). 
 J\'er\'emy Blanc proves in~\cite{Blanc} that
 \begin{itemize}
 \item there are no relations between $s_Y$ and $s'_Y$: the group they generate is isomorphic to $\Z/2\Z\star \Z/2\Z$.
 \item the composition $s_Y\circ s'_Y$ is a parabolic automorphism (hence a Halphen twist).
 \end{itemize}
 Thus, to construct examples of Halphen twists in characteristic $\neq 2$, we only need to construct a smooth
 cubic curve and two points $q$, $q'$ of $C(\bfk)$. If $\fieldchar(\bfk)\neq 7$, one can take the curve 
defined by 
$x^3+y^3+z^3+xyz=0$, with $q=[1:-1:0]$ and $q'=[1:0:-1]$. If $\fieldchar(\bfk)= 7$, one can take the curve
$y^2z=x^3+z^3$, with $q=[-1:0:1]$ and $q'=[0:1:0]$.

\subsubsection{Characteristic $2$ } Let us describe another strategy that works in characteristic $2$. Consider a field $\bfk$ of characteristic $2$. 
 In $\P^1\times \P^2$, with affine coordinates $(t)\times (x,y)$, consider the surface $X\subset \P^1\times \P^2$ given by the equation 
 \begin{equation}
 y^2+xy+t^3y=x^3+1.
 \end{equation}
Let $\pi\colon X\to \P^1$ be the first projection. For $t\neq \infty$, the fiber $F_t=\pi^{-1}(t)$ is a cubic curve
in Weierstrass form. There are two sections of $\pi$, given by $t\mapsto (-1,0)=[-1:0:1]$ and $t\mapsto [0:1:0]$.
Let $f\colon X\to X$ be the birational transformation of $X$ that preserves each of the fibers of $\pi$ and translates
the first section to the second. 
In~\cite{WLang}, William E. Lang shows that this Weierstrass pencil is a Halphen pencil with twelve singular
irreducible curves. Thus, there is no $-2$ curve on $X$, and this implies that $f$ is a Halphen twist (see~\cite{Cantat-Dolgachev}, Theorem~2.10).

\subsubsection{An example over $\C$} Our last example shows that $n_0$ must be larger than $1$ in Theorem~B  and~B'.

\begin{pro}
There is a Halphen twist $f$ in $\Cr_2(\C)$ such that 
\begin{enumerate}
\item $f$ preserves
every member of its invariant pencil and 
\item the smallest normal subgroup of $\Cr_2(\C)$ that contains $f$ is equal
  to $\Cr_2(\C)$.
\end{enumerate}

\end{pro}

Note that (1) implies that $f$ acts as a translation on the general member of its invariant pencil. 

\begin{proof} 
Consider the  
ring of Eisenstein integers $\Lambda:=\Z[\jj]\subset \C$, where $\jj$ a primitive cubic root of unity. Denote by $E$ 
the elliptic curve $\C/\Lambda$ and by $A$ the abelian surface $E\times E$. 
The two matrices 
\begin{equation}
F=\left( 
\begin{array}{lr} 1 & 1 \\0 & 1 \end{array}
\right) \; {\text{ and }} \; G=\left( 
\begin{array}{lr} 1 & 0 \\1 & 1 \end{array}
\right)
\end{equation}
generate $\SL_2(\Z) \subset \GL_2(\Lambda)$, and in particular the order $3$ element 
\begin{equation}
U=\left( \begin{array}{lr} 0 & -1 \\1 & -1 \end{array}\right).
\end{equation}
Denote by $\Pi_F$ and $\Pi_G$ the fibration $A\to E$ given by the projection onto the second and the first factors. Since
$\GL_2(\Lambda)$ preserves the lattice $\Lambda\times \Lambda$ of $\C^2$, the group $\GL_2(\Lambda)$ embeds into
the group of automorphisms of $A$. The automorphism $F$ (resp. $G$) is a Halphen twist with respect to the fibration $\Pi_F$ 
(resp. $\Pi_G$)(\footnote{See~\cite{Cantat:Survey} for Halphen twists on any surface. Here, this corresponds to the following equivalent facts: (a) the action of $F$ on the cohomology group $H^{1,1}(A;\R)$ satisfies $\parallel (F^n)^*\parallel\simeq \alpha n^2$ for some $\alpha>0$, or (b) the action of $F$ on the N\'eron-Severi group $\NS(A)$  satisfies $\parallel (F^n)^*\parallel\simeq \alpha n^2$, or (c) given any polarization $H$ on $A$, the degree $H\cdot (F^n)^*H$ grows like $\alpha n^2$. }). Taking the quotient of $A$ by  $J(x,y)=(\jj x, \jj y)$ (an element of the center of $\GL_2(\Lambda)$), 
we get a (singular) surface $X_0$ on which $\PGL_2(\Lambda)$ acts faithfully. The surface $X_0$ is a rational surface: there 
is a birational map $\varphi\colon X_0\to \P^2_\C$ (see~\cite[Lemma 4.1]{Grivaux_rational_surface_auto}). 

The automorphisms $F$, $G$, and $U$  determine automorphisms $f_0$, $g_0$, and $u_0$ of $X_0$; after conjugacy 
by  $\varphi$, we get three elements $f$, $g$, $u$ of $\Cr_2(\C)$ such that: 
\begin{enumerate}
\item $f$ and $g$ are Halphen twists (with respect to distinct pencils);
\item $u$ has order $3$, and is contained in the group generated by $f$ and $g$. 
\end{enumerate}
Moreover, the fixed point set of $U$ is a finite subset of $A$, hence the fixed point set of $u_0$ 
is finite. The singularities of $X_0$ are quotient singularities, and their resolution only creates new rational curves. 
So, the fixed point set of $u$ does not contain any irrational curve, and this 
implies that $u$ is conjugate to an element of $\PGL_3(\C)$ (see~\cite{Blanc_linearisation_cyclic}). By Lemma~\ref{lem:elliptic} the smallest normal subgroup generated by 
$f$ and $g$ coincides with $\Cr_2(\C)$. 

Also, the linear map $(x,y)\mapsto (y,x)$ conjugates $F$ to $G$. So, the smallest normal subgroup generated by $f$
contains $g$, and the proposition is proven. 
\end{proof}

\bibliographystyle{plain}
\bibliography{biblio}

\end{document}